\documentclass[12pt, reqno]{amsart}
\usepackage{amsmath, amsthm, amscd, amsfonts, amssymb, graphicx, color, mathrsfs}
\usepackage[bookmarksnumbered, colorlinks, plainpages]{hyperref}

\usepackage{cancel}
\newtheorem{theorem}{Theorem}[section]
\newtheorem{lemma}[theorem]{Lemma}
\newtheorem{proposition}[theorem]{Proposition}
\newtheorem{corollary}[theorem]{Corollary}
\theoremstyle{definition}
\newtheorem{definition}[theorem]{Definition}

\theoremstyle{remark}
\newtheorem{remark}[theorem]{Remark}
\numberwithin{equation}{section}

\newcommand*\diff{\mathop{}\!\mathrm{d}}

\usepackage{bbm}

\newcounter{casenum}

\begin{document}
\setcounter{page}{1}

\title[Weak (1,1)-Boundedness of operators on compact Lie groups]{Kernel estimates and weak (1,1)-boundedness of pseudo-differential operators on compact Lie groups}

 \author[D. Cardona]{Duv\'an Cardona}
\address{
 Duv\'an Cardona:
  \endgraf
  Department of Mathematics: Analysis, Logic and Discrete Mathematics
  \endgraf
  Ghent University, Belgium
  \endgraf
  {\it E-mail address} {\rm duvan.cardonasanchez@ugent.be}
  }

  \author[R. Yeghoyan]{Rafik Yeghoyan}
\address{
 Rafik Yeghoyan:
  \endgraf
  Department of Mathematics: Analysis, Logic and Discrete Mathematics
  \endgraf
  Ghent University, Belgium
  \endgraf
  {\it E-mail address} {\rm rafik.yeghoyan@ugent.be}
  }

\author[M. Ruzhansky]{Michael Ruzhansky}
\address{
  Michael Ruzhansky:
  \endgraf
  Department of Mathematics: Analysis, Logic and Discrete Mathematics
  \endgraf
  Ghent University, Belgium
  \endgraf
 and
  \endgraf
  School of Mathematical Sciences
  \endgraf
  Queen Mary University of London
  \endgraf
  United Kingdom
  \endgraf
  {\it E-mail address} {\rm michael.ruzhansky@ugent.be}
  }

 \allowdisplaybreaks



\subjclass[2010]{}

\thanks{ The authors are supported  by the FWO  Odysseus  1  grant  G.0H94.18N:  Analysis  and  Partial Differential Equations and by the Methusalem programme of the Ghent University Special Research Fund (BOF)
(Grant number 01M01021). Michael Ruzhansky is supported by the FWO Senior Research Grant G011522N and by the EPSRC grant UKRI3645. Duv\'an Cardona was supported by the Research Foundation-Flanders
(FWO) under the postdoctoral grant No 1204824N}
\date{2025}

\begin{abstract} 
Given a compact Lie group \( G \) and its unitary dual $\widehat{G}$, we establish the weak (1,1) continuity for pseudo-differential operators in the global H\"ormander classes of order $-n(1-\rho)/2$ on $G\times \widehat{G}$. Our approach consists of proving suitable estimates for the kernel of such operators. Furthermore, we use these kernel estimates to give an alternative proof for the $H^1(G)$-$L^1(G)$-continuity of these classes now allowing the full range $0\leq\delta\leq\rho\leq1, \;\rho\neq0,\;\delta\neq1$. The conditions for the operators are formulated using the H\"ormander classes $S^m_{\rho,\delta}(G):=S^m_{\rho,\delta}(G\times \widehat{G})$ of symbols in the non-commutative phase space $G\times \widehat{G}$, which are extensions of the well-known $(\rho,\delta)$-classes in the Euclidean space. Our results are formulated in the complete range $0\leq \delta\leq \rho\leq 1,$ $\rho\neq0,\;$$\delta\neq 1$. 

As an application of this boundedness result we provide {\it end-point a-priori $L^1$-estimates} for the sub-Laplacian $\mathcal{L}_{sub}=X^2+Y^2,$ and for the heat type operator $T=Z-X^2-Y^2$  on $SU(2)\cong \mathbb{S}^3$ that cannot be obtained by application of the standard pseudo-differential calculus due to H\"ormander. More precisely, we prove that if one considers the subelliptic problem,
    \begin{equation}\label{IVP:abstract} \begin{cases}Tu=f ,& \text{ }
\\u,f\in \mathscr{D}'(SU(2)):=(C^\infty(SU(2)))', & \text{ } \end{cases}
\end{equation} then, for $f\in W^{1,-\frac{1}{4}}(SU(2)),$ one has that $u\in L^{1,\infty}(SU(2)).$
\end{abstract} \maketitle

\tableofcontents
\allowdisplaybreaks
\section{Introduction}

\subsection{Overview} The mathematical research regarding the boundedness of pseudo-differential operators is a fundamental topic of the harmonic analysis after pioneering works on the subject of, e.g., H\"ormander \cite{13,14} and Fefferman \cite{11}. Indeed, it plays an important role in applications to the well-posedness and stability properties of partial differential equations on $\mathbb{R}^n$ and other Riemannian manifolds. Focused on the $L^1$-theory,  this paper studies the weak (1,1) continuity of pseudo-differential operators on an arbitrary compact Lie group $G$. This question goes back to the work \cite{9}.  We also provide alternative proofs for the boundedness of pseudo-differential operators from the Hardy space $H^1(G)$ to $L^1(G)$, from $L^\infty(G)$ to $BMO(G)$, and from $L^p(G)$ to itself. These results were previously established in \cite{9} and for subelliptic pseudo-differential operators in \cite{CR}. Our arguments rely on kernel estimates and are therefore considerably simpler, also allowing us to obtain the weak $(1,1)$ results.

Here, we recall that the classes $S^m_{\rho,\delta}(G)$ on compact Lie groups have been introduced by the third author and Turunen in \cite{17}, with full generality on the parameters $\rho$ and $\delta$ by satisfying $0\leq \delta\leq \rho\leq 1,$ $\rho\neq 0$,\;$\delta\neq 1.$ We denote by $\Psi^m_{\rho,\delta}(G)$ the class of operators with symbols in the class $S^m_{\rho,\delta}(G)$. We would like to note that when we understand $G$ as a compact manifold without boundary, symbol classes can be defined using local coordinate systems in terms of the classical classes $S^m_{\rho,\delta}(\mathbb{R}^n\times\mathbb{R}^n)$ introduced by H\"ormander \cite{13,14}. However, on compact manifolds we have to impose the restriction $\delta<\rho$, and $\rho\geq 1-\delta$. The global quantisation of pseudo-differential operators on compact Lie groups in \cite{17} is introduced by relying on the representation theory of the group rather than on expressions in local coordinates. In Section \ref{Section:2} we give a short overview about the global notion of a symbol on compact Lie groups. 

In our approach, the continuity of pseudo-differential operators follows from suitable kernel estimates combined  with the $L^2$-estimates \cite{12} as established in the Euclidean setting.  In $\mathbb{R}^n,$ a classical boundedness result due to Charles Fefferman \cite{11} established the $L^p$-continuity of pseudo-differential operators with symbols in the H\"ormander classes $\Psi^m_{\rho,\delta}(\mathbb{R}^n)$ for  $\delta<\rho$. However, this estimate was extended e.g. by Alvarez and Hounie in \cite{1} allowing the case where $0<\rho\leq 1$, $0\leq\delta<1$, and a suitable order of the operator depending on these parameters. Their results can be  extended to general compact manifolds without boundary for $\rho\geq \max\{\delta,1-\delta\}$.  

 In compact Lie groups, Delgado and the third author addressed this problem in \cite{9}, investigating the $L^p$-boundedness, the continuity from $L^\infty(G) $ to $BMO(G)$, and from $H^1(G)$ to $L^1(G)$, of pseudo-differential operators when $\delta<\rho$. According to the analysis in  \cite{9}, the condition $\rho\geq 1-\delta$ was not required. As it was mentioned before, these results were extended in \cite{CR} for $\delta\leq \rho,$ for subelliptic pseudo-differential operators. Our article aims to prove the weak (1,1) inequality for H\"ormander classes on compact Lie groups in the range $0\leq \delta\leq \rho\leq 1,$ $\rho\neq 0$,\;$\delta\neq 1$. In other words, we want to investigate the weak (1,1) continuity for these classes.
For recent results in the case of the torus, we refer to \cite{Cardona2014,Cardona:Martinez}, with the results in \cite{Cardona:Martinez} admitting sharp criteria for the order in terms of $\rho,\delta\in [0,1]$ and of $\lambda=\max\{0,\frac{\delta-\rho}{2}\}$, where $\lambda$ represents the sharp correction to the operator’s order needed for weak (1,1)-continuity.
 
 \subsection{Methodology} In order 
to provide an idea of our approach regarding the weak (1,1) continuity property, it is interesting to recall the argument by Álvarez and Milman that obtains a weak $(1,1)$ bound for Calderón-Zygmund operators, see \cite[Theorem 4.1]{2}. Given $f\in L^1(\mathbb{R}^n)$ and $\alpha>0$ arbitrarily, using the Calderón–Zygmund decomposition, we can write $f=g+b$. Here $b=\sum_{j=1}^{\infty}b_j$, where the $b_j$'s are supported on disjoint dyadic cubes $I_j$'s such that 
\begin{equation*}
    \frac{1}{|I_j|}\int_{I_j}|b_j(x)|\diff{x}\sim\alpha,\hspace{5pt}\int_{I_j}b_j(x)\diff{x}=0.
\end{equation*}
Then we can estimate
\begin{eqnarray*}
     &|\{x\in\mathbb{R}^n:|Tf(x)|>\alpha\}|\leq|\{x\in\mathbb{R}^n:|Tg(x)|>\alpha/2\}|\\
     &+|\{x\in\mathbb{R}^n:|Tb(x)|>\alpha/2\}|.
\end{eqnarray*}
The part regarding $T(g)$ can easily be estimated using the fact that $g\in L^2(\mathbb{R}^n)$ and the extension of $T$ to $L^2(\mathbb{R}^n)$. The part involving $T(b)$ is where most of the work is done. First we further decompose $b=F'+F''$, the estimation for $TF''$ follows from the condition that one has on the kernel of such operators. For $TF'$ we can use Fefferman's argument developed in \cite{10}, where the idea is to show that
\begin{equation*}
    \|J^\beta F'\|^{2}_{L^2}\lesssim\alpha\|f\|_{L^1}.
\end{equation*}
Here $J^\beta$ denotes the Bessel potential of order $\beta$. Using the $L^2$ continuity of $T$, we see that
\begin{equation*}
    \|TF'\|^2_{L^2}=\|TJ^{-\beta}J^\beta F'\|^{2}_{L^2}\lesssim\|J^\beta F'\|^{2}_{L^2}\lesssim\alpha\|f\|_{L^1}.
\end{equation*}
The result follows from an application of Chebyshev's theorem. We will adopt a similar approach, but in the context of compact Lie groups. For extension of the methods of \'Alvarez and Hounie \cite{1} and \'Alvarez and Milman \cite{2} to the torus $\mathbb{T}^n\sim \mathbb{R}^n/\mathbb{Z}^n $ we refer to \cite{Cardona:Martinez}.
\subsection{Main results}
  Here we present our main theorems. We begin with a key kernel estimate, which will play a central role in the proof of weak (1,1) continuity. This estimate is stated in the form of the following result where the main part is the estimate for $R<1$. 
 \begin{theorem}\label{kernel_estimates}
    Let $G$ be a compact Lie group of topological dimension $n\geq 1$. Consider $T\in\Psi^{m}_{\rho,\delta}(G)$, with \[
m \le -\frac{n}{2}(1-\rho),
\qquad 0 \le \delta \le \rho \le 1,\ \rho\neq0,\;\delta \neq 1.
\]
For $R>0$, we have that the kernel of $T$ satisfies the following. For $y\in\overline{B(z, R)}$ we have
    \begin{itemize}
        \item $R\geq 1$ then
        $\int_{B(z, 2R)^c}|K(x,y)-K(x,z)|\diff{x}\leq C$;
         \item $R<1$ then $\int_{B(z, 2R^{\rho})^c}|K(x,y)-K(x,z)|\diff{x}\leq C$;
    \end{itemize}
    Here the constant $C>0,$ in the above estimates is independent of $R>0.$
\end{theorem}
As an application of Theorem \ref{kernel_estimates} we present the following weak (1,1) estimate.
\begin{theorem}\label{weak}
    Let $G$ be a compact Lie group of topological dimension $n\geq 1$. Let $T\in\Psi^{m}_{\rho,\delta}(G)$, with \[
m \le -\frac{n}{2}(1-\rho),
\qquad 0 \le \delta \le \rho \le 1,\ \rho\neq0,\;\delta \neq 1.
\]
Then, $T$ is a continuous mapping from $L^1(G)$ to $L^{1,\infty}(G)$.
\end{theorem} 
As an application of our results we provide {\it end-point a-priori $L^1$-estimates} for the sub-Laplacian $\mathcal{L}_{sub}=X^2+Y^2,$ and for the heat operator $T=Z-X^2-Y^2$  on $SU(2)\cong \mathbb{S}^3,$ see Section \ref{applications}.

 We note that Theorem \ref{H1:L1} and Theorem \ref{L:infty:BMO} in Section \ref{Section:3} have already been proved in \cite{CR} for $\delta\leq \rho,$ following the approach in \cite{9} where the case $\delta< \rho,$ was considered. However, the novelty of this paper in relations to these results comes from the kernel estimates proved in this paper, from which we substantially simplify these proofs. To prove Theorem \ref{H1:L1}, we will use a method similar to that of Alvarez and Hounie \cite[Theorem 3.2]{1} in the context of $\mathbb{R}^n$. The idea of their proof is to show that for each $(1,\infty)$ atom $a$, we have $\|Ta\|_{L^1}\lesssim 1$. We do this as follows; by the definition of an atom, there exists a ball $B(z,R)$ such that
  \begin{equation*}
        \operatorname{supp}(a)\subseteq B(z,R),\hspace{5pt}\|a\|_{L^\infty}\leq|B(z,R)|^{-1},\hspace{5pt}\int_{B(z,R)}a(x)\diff{x}=0.
 \end{equation*}
Then,
\begin{equation*}
\int_{\mathbb{R}^n}|Ta(x)|\diff{x}\leq\int_{B(z,2R^\rho)}|Ta(x)|\diff{x}+\int_{\mathbb{R}^n\setminus B(z,2R^\rho)}|Ta(x)|\diff{x}\equiv I_1+I_2.
    \end{equation*} 
Now the idea is to estimate both integrals separately. To estimate $I_1$, an analogue of the following theorem is needed (see  \cite[Theorem 3.1]{1}).
\begin{theorem}
    Given $T\in\Psi^{m}_{\rho,\delta}(\mathbb{R}^n)$, $0<\rho\leq 1$, $0\leq\delta<1$, $m\leq-n\left(\frac{1-\rho}{2}+\lambda\right)$, where $\lambda=\max\{0,\frac{\delta-\rho}{2}\}$. Then $T$ is continuous from $L^2$ into $L^\frac{2}{\rho}$ and from $L^{\frac{2}{2-\rho}}$ to $L^2$.  
\end{theorem}
We will need to prove an analogous theorem in the context of compact Lie groups; for details, see Section \ref{Section:3}. The estimation for $I_2$ follows from the kernel condition. Theorem \ref{L:infty:BMO} follows from Theorem \ref{H1:L1} by the duality argument and the Fefferman-Stein duality property $(H^1(G))^{'}=BMO(G)$, which we explain in Section \ref{Section:2}. Note that we can interpolate the $L^2$-boundedness of the operator $T$ with its $L^\infty$-$BMO$ and $H^1$-$L^1$-boundedness in order to obtain the following $L^p$-boundedness property for operators of order $m=-\frac{n}{2}(1-\rho).$ The $L^{p}$-boundedness of operators in $\Psi_{\rho,\delta}^{m}(G)$ has also been
studied in \cite{9}, for the range $0 \leq \delta < \rho \leq 1$. They also provided $L^p$-boundedness for the H\"ormander classes on compact Lie groups with orders in the interval $[-\frac{n}{2}(1-\rho),0]$ only when the Lebesgue index $p$ belongs to small intervals centered at $p_0=2$  in the following way,  see \cite[Theorem 4.15]{9}. 
\begin{theorem}\label{Lp:DR}
    Let $G$ be a compact Lie group of dimension $n$ and $0<\rho\leq 1$. Let $0\leq\delta<\rho$ and consider a symbol $\sigma:=\sigma(x,\xi)$ in the H\"ormander class $S^{m}_{\rho,\delta}(G).$
    Then $\sigma(x,D)$ extends to a bounded operator from $L^p(G)$ to $L^p(G)$ for all $1<p<\infty$ such that 
    \begin{equation*}
    m\leq   -n(1-\rho)  \left|\frac{1}{p}-\frac{1}{2}\right|.
    \end{equation*}Moreover, for $m\leq -n(1-\rho)/2,$ $\sigma(x,D)$  is bounded from the Hardy space $H^1(G)$ to $L^1(G)$ and from $L^\infty(G)$ to $BMO(G).$
\end{theorem}
The index $m:=\frac{n(1-\rho)}{2}$ in Theorem \ref{Lp:DR} is sharp, see \cite[Remark 4.16]{9} for more details. 
\begin{remark} It is interesting to mention that our kernel estimates in Theorem \ref{kernel_estimates} also provide a new approach to prove Theorem \ref{Lp:DR} for the full range $0 \leq \delta \le \rho \le 1, \rho\neq0,\,\delta \neq 1$, (see Theorem \ref{H1:L1} and Theorem \ref{L:infty:BMO}).
    
\end{remark}

\section{Preliminaries}\label{Section:2}
\subsection{The group Fourier transform}
We will now recall some basic facts about the theory of pseudo-differential operators on compact Lie groups, and we refer to \cite{17} and \cite{18} for a comprehensive study. Let $G$ be a compact Lie group and $\diff{x}$ its normalized left-invariant Haar measure. The Lie algebra of $G$ is denoted by $\mathfrak{g}$. Let $\widehat{G}$ denote the unitary dual of $G$, i.e. the set of equivalence classes of irreducible unitary representations of $G$. Let $[\zeta]\in\widehat{G}$ denote the equivalence class of an irreducible unitary representation $\zeta\colon G\to\operatorname{\mathcal{U}}(\mathcal{H}_\zeta)$; the representation space $\mathcal{H}_\zeta$ is finite-dimensional since $G$ is compact, and we set $d_\zeta:=\operatorname{dim}(\zeta)=\operatorname{dim}(\mathcal{H}_\zeta)$. We define the Fourier coefficient $\widehat{f}(\zeta)\in\operatorname{End}(\mathcal{H}_\zeta)$ of $f\in L^1(G)$ by
\begin{equation*}
      \mathscr{F}f(x)\equiv  \widehat{f}(\zeta)=\int_{G}f(x)\zeta(x)^*\diff{x},
    \end{equation*}
    more precisely, 
    \begin{equation*}
        (\widehat{f}(\zeta)u,v)_{\mathcal{H}_\zeta}=\int_{G}f(x)(\zeta(x)^*u,v)_{\mathcal{H}_\zeta}\diff{x}=\int_{G}f(x)(u,\zeta(x)v)_{\mathcal{H}_\zeta}\diff{x},
    \end{equation*}
    for every $u,v\in\mathcal{H}_{\zeta}$, where $(\cdot,\cdot)_{\mathcal{H}_{\zeta}}$ denotes the inner product on $\mathcal{H}_{\zeta}$.
The Fourier inversion formula follows from the Peter Weyl theorem, we have that
\begin{equation*}
    f(x)=\mathscr{F}^{-1}(\widehat{f}\,)(x):=\sum_{[\zeta]\in\widehat{G}}d_\zeta\operatorname{Tr}(\zeta(x)\widehat{f}(\zeta)).
\end{equation*}
The Parseval identity takes the form
\begin{equation*}
    \|f\|_{L^2(G)}=\left(\sum_{[\zeta]\in\widehat{G}}d_\zeta\|\widehat{f}(\zeta)\|_{HS}^2\right)^{\frac{1}{2}},
\end{equation*}
here $\|\widehat{f}(\zeta)\|_{HS}^2=\operatorname{Tr}(\widehat{f}(\zeta)\widehat{f}(\zeta)^*)$. 

\subsection{Global pseudo-differential operators}
Let $A\colon C^{\infty}(G)\to\mathcal{D}'(G)$ be continuous and linear. We can define its matrix-valued symbol $\sigma_A(x,\zeta)\in\mathbb{C}^{d_\zeta\times d_\zeta}$ by
\begin{equation*}
    \sigma_A(x,\zeta)=\zeta(x)^*(A\zeta)(x),
\end{equation*}
 where $A\zeta(x)\in\mathbb{C}^{d_\zeta\times d_\zeta}$ is understood as $(A\zeta(x))_{i,j}=(A\zeta_{ij})(x)$. Then one has the global quantization
 \begin{equation*}
        Af(x)=\sum_{[\zeta]\in\widehat{G}}d_\zeta\operatorname{Tr}(\zeta(x)\sigma_A(x,\zeta)\widehat{f}(\zeta)),
    \end{equation*}
in the sense of distributions, and the sum is independent of the choice of a representation $\zeta$ from each equivalence class $[\zeta]\in\widehat{G}$.
\begin{definition}
     Let $\{Y_j\}_{j=1}^{n}$ be a basis for $\mathfrak{g}$. Denote by $\partial_j$ the left-invariant first-order differential operator corresponding to $Y_j$. For $\alpha\in\mathbb{N}_{0}^{n}$, we denote $\partial^\alpha=\partial^{\alpha_1}_{1}\cdots\partial^{\alpha_n}_{n}$. This is also often denoted by $\partial_x^\alpha$.
\end{definition}
 We say that $Q_\zeta$ is a difference operator of order $k$ if it is given by
 \begin{equation*}
     Q_\zeta\widehat{f}(\zeta)=\widehat{q_Qf}(\zeta),
 \end{equation*}
 for a function $q=q_Q\in C^{\infty}(G)$  vanishing of order $k$ at the identity $e\in G$, i.e. $(P_xq_Q)(e)=0$ for all left-invariant differential operators $P_x\in\operatorname{Diff}^{k-1}(G)$ of order $k-1$. The set of all difference operators of order $k$ is denoted by $\operatorname{diff}^k(G)$. For a given function $q\in C^{\infty}(G)$ we denote the associated difference operator, acting on Fourier coefficients by
 \begin{equation*}
     \triangle_q\widehat{f}(\zeta)=\widehat{qf}(\zeta).
 \end{equation*}
\begin{definition}
    A collection of $m$ first-order difference operators $\triangle_1,\dots,\triangle_m\in\operatorname{diff}^1(\widehat{G})$ is called admissible, if the corresponding functions $q_1,\dots,q_m\in C^{\infty}(G)$ satisfy $d{q}_j(e)\neq 0$, $j=1,\dots,m$ and $\operatorname{rank}(d{q}_1(e),\dots,d{q}_m(e))=n$. In particular, it follows that $e$ is an isolated common zero of the family $\{q_j\}_{j=1}^{m}$. An admissible collection is called strongly admissible if 
    \begin{equation*}
        \bigcap_{j}\{x\in G: q_j(x)=0\}=\{e\}.
    \end{equation*}
 \end{definition}
Here, we consider the inner product $$g(X,Y)=(X,Y)_{g}:=-\textnormal{Tr}[\textnormal{ad}(X)\textnormal{ad}(Y)],$$ on $\mathfrak{g}$ for non-commutative $G$, where
\(\operatorname{ad}(X) : \mathfrak{g} \to \mathfrak{g}\) is given by
\(\operatorname{ad}(X)(Z) = [X,Z]\). Note that $(\cdot,\cdot)_{g}$ is the negative of the Killing form.  Let $\mathbb{X}=\{X_1,\cdots,X_n\}$ be an orthonormal basis of $\mathfrak{g}$ with respect to $(\cdot,\cdot)_{g}$. The canonical positive Laplacian on $G$ is defined via
$
    \mathcal{L}_G=-\sum_{j=1}^nX_j^2,
$ 
and then is independent of the choice of  $\mathbb{X}.$ We also use the multi-index notation $\triangle^\alpha_\zeta=\triangle^{\alpha_1}_{1}\cdots\triangle^{\alpha_m}_m$. A matrix-valued symbol $\sigma(x,\zeta)$ belongs to $S^m_{\rho,\delta}(G)$ if it is smooth in $x$, and for all multi-indices $\alpha,\beta$ there exists a constant $C_{\alpha,\beta}>0$ such that
\begin{equation*}
    \|\triangle_{\zeta}^\alpha\partial^{\beta}_{x}\sigma(x,\zeta)\|_{op}\leq C_{\alpha,\beta}\langle\zeta\rangle^{m-\rho|\alpha|+\delta|\beta|},
\end{equation*}
holds uniformly in $x$ and $\zeta\in\operatorname{Rep}(G)$. Here $\langle\zeta\rangle=(1+\lambda_{[\zeta]})^{\frac{1}{2}}$, and $\{\lambda_{[\zeta]}\}_{[\zeta]\in\widehat{G}}$ is the positive spectrum of the Laplacian $\mathcal{L}_{G}$. We also have employed the notation $\|\cdot\|_{op}$ for the operator norm of matrices on each representation space induced by the Euclidean norm. Next, we present several theorems that will be essential for proving the main results. The first proposition will be essential in proving the kernel estimation for large radii. 
\begin{proposition}\label{kernel}
    Let $\sigma\in S^m_{\rho,\delta}$ with $0 \leq\delta\leq\rho\leq 1$, $\rho\neq0$. Then its associated kernel $(x,y)\to k_x(y)\in C^{\infty}(G\times G\setminus{e_G})$ satisfies the following kernel estimates:
    \begin{itemize}
        \item If $n+m>0$ then there exist $C$ and $a,b\in\mathbb{N}$ (independent of $\sigma$) such that
        \begin{equation*}
            |k_x(y)|\leq C\sup_{\pi\in\widehat{G}}\|\sigma(x,\pi)\|_{S^{m}_{\rho,a,b}}|y|^{-\frac{n+m}{\rho}}.
        \end{equation*}
        \item If $n+m=0$ then there exist $C$ and $a,b\in\mathbb{N}$(independent of $\sigma$) such that
        \begin{equation*}
            |k_x(y)|\leq C\sup_{\pi\in\widehat{G}}\|\sigma(x,\pi)\|_{S^{m}_{\rho,a,b}}|\ln|y||.
        \end{equation*}
        \item If $n+m<0$ then $k_x$ is continuous on $G$ and bounded
        \begin{equation*}
            |k_x(y)|\lesssim_{m}\sup_{\pi\in\widehat{G}}\|\sigma(x,\pi)\|_{S^{m}_{\rho,0,0}}.
        \end{equation*}
    \end{itemize}
\end{proposition}
\begin{proof}
    See \cite[Proposition 6.7]{12}. 
\end{proof}
The following theorem will be useful in the proof of the kernel estimation for small radii. 
\begin{theorem}\label{Weyl}
    Let $\alpha\in\mathbb{R}$ with $\alpha\neq-1$. Then we have that the following asymptotic properties hold as $\lambda\to\infty$
    \begin{equation*}
    \begin{aligned}
        &\sum_{\langle\zeta\rangle\leq\lambda}d_\zeta^2\langle\zeta\rangle^{\alpha n}\asymp\lambda^{(\alpha+1)n}\quad\mathrm{for}\,\,\alpha>-1,\\
        &\sum_{\langle\zeta\rangle\geq\lambda}d_\zeta^2\langle\zeta\rangle^{\alpha n}\asymp\lambda^{(\alpha+1)n}\quad\mathrm{for}\,\,\alpha<-1.
    \end{aligned}
    \end{equation*}
\end{theorem}
\begin{proof}
    See \cite[Section 5]{Weyl}. 
\end{proof}
Finally, we present two results that will help us to prove the weak (1,1)-boundedness theorem. 
\begin{proposition}\label{L2}
    Let $0 \leq\delta\leq\rho\leq 1$ with $\delta\neq 1$. If $\sigma\in S^0_{\rho,\delta}$ then $\operatorname{Op}(\sigma)$ is bounded on $L^2(G)$.
\end{proposition}
\begin{proof}
    See \cite[Proposition 8.1]{12}.
\end{proof}

\begin{theorem}[Calderón-Zygmund decomposition]\label{CZ}
    Let $f\in L^1(G)$ and $\gamma,\alpha>0$. Then for 
    \begin{equation*}
        \alpha\gamma>\frac{1}{|G|}\int_{G}|f(x)|\diff{x},
    \end{equation*}
    we can decompose $f=g+b=g+\sum_{j}b_j$ where 
    \begin{enumerate}
        \item $\|g\|_{L^\infty}\lesssim\gamma\alpha$ and $\|g\|_{L^1}\lesssim\|f\|_{L^1}$.
        \item $\operatorname{supp}(b_j)\subseteq I_j=B(x_j,r_j)$ where 
        \begin{equation*}
            \int_{I_j}b_j(x)\diff{x}=0.
        \end{equation*}
        \item Any component $b_j$ satisfies the $L^1$-estimate
        \begin{equation*}
            \|b_j\|_{L^1}\lesssim(\gamma\alpha)|I_j|.
        \end{equation*}
        \item The sequence $\{|I_j|\}_{j}\in l^{1}$ and 
        \begin{equation*}
            \sum_{j}|I_j|\lesssim\frac{\|f\|_{L^1}}{\gamma\alpha}.
        \end{equation*}
        \item 
        \begin{equation*}
            \|b\|_{L^1}\leq\sum_{j}\|b_j\|_{L^1}\lesssim\|f\|_{L^1}.
        \end{equation*}
        \item There exists $M\in\mathbb{N}$, such that every $x\in G$ belongs to at most to $M$ balls of the collection $\{I_j\}_j$.
    \end{enumerate}
\end{theorem}
\begin{proof}
    See \cite{7}.
\end{proof}
\begin{remark}\normalfont
We denote by $\Omega=\{x\in G: m_f(x)>\alpha\gamma\}$, where 
    \begin{equation*}
        m_f(x)=\sup_{B \ni x}\frac{1}{|B|}\int_{B}|f(y)|\diff{y}.
    \end{equation*}
Here $B$ ranges over balls that contain $x$. From the Calderón-Zygmund decomposition we have $\Omega=\bigcup_{j=1}^{\infty}B(x_j,r_j)$.
\end{remark}
\subsection{The spaces $BMO$ and $H^1$  on compact Lie groups}
Now we introduce the $BMO$ space on a compact Lie group as well as the Hardy space $H^1(G).$  We will fix the geodesic distance $|\cdot|$ on $G.$  As usual, the ball of radius $r>0,$ is defined as 
\[ 
    B(x,r)=\{y\in G:|y^{-1}x|<r\}.
\]Then the $BMO$ space on $G,$ which we denote by ${BMO}(G),$ is the space of locally integrable functions $f$ satisfying
\[ 
    \Vert f\Vert_{{BMO}(G)}:=\sup_{\mathbb{B}}\frac{1}{|\mathbb{B}|}\int\limits_{\mathbb{B}}|f(x)-f_{\mathbb{B}}|dx<\infty,\textnormal{ where  } f_{\mathbb{B}}:=\frac{1}{|\mathbb{B}|}\int\limits_{\mathbb{B}}f(x)dx,
\]
and $\mathbb{B}$ ranges over all balls $B(x_{0},r),$ with $(x_0,r)\in G\times (0,\infty).$ 

On the other hand, the Hardy space ${H}^{1}(G)$ will be defined via the atomic decomposition. Thus, $f\in {H}^{1}(G)$ if and only if $f$ can be expressed as $$f=\sum_{j=1}^\infty c_{j}a_{j},$$ where $\{c_j\}_{j=1}^\infty$ is a sequence in $\ell^1(\mathbb{N}),$ and every function $a_j$ is an atom, i.e., $a_j$ is supported in some ball $B_j,$ $a_j$ satisfies the cancellation property $$\int\limits_{B_j}a_{j}(x)dx=0,$$ and 
\[ 
    \Vert a_j\Vert_{L^\infty(G)}\leqslant \frac{1}{|B_j|}.
\] The norm $\Vert f\Vert_{{H}^{1}(G)}$ is the infimum over  all possible series $\sum_{j=1}^\infty|c_j|.$ Furthermore ${BMO}(G)$ is the dual space of ${H}^{1}(G)$. This can be understood in the following sense:
\begin{itemize}
    \item[(a)] if $\phi\in {BMO}(G), $ then $$\Phi: f\mapsto \int\limits_{G}f(x)\phi(x)dx,$$ admits a bounded extension on ${H}^{1}(G).$
    \item[(b)] Conversely, every continuous linear functional $\Phi$ on ${H}^{1}(G)$ arises as in $\textnormal{(a)}$ with a unique element $\phi\in {BMO}(G).$
\end{itemize} The norm of $\phi$ as a linear functional on ${H}^{1}(G)$ is equivalent with the ${BMO}(G)$-norm. Important properties of the ${BMO}(G)$ and the ${H}^{1}(G)$ norms are the following,
\begin{equation}\label{BMOnormduality}
 \Vert f \Vert_{{BMO}(G)}  =\sup_{\Vert g\Vert_{{H}^{1}(G)}=1} 
\left| \int\limits_{G}f(x)g(x)dx\right|,\end{equation}
\begin{equation}\label{BMOnormduality'}\Vert g \Vert_{{H}^{1}(G)}  =\sup_{\Vert f\Vert_{{BMO}(G)}=1} 
\left| \int\limits_{G}f(x)g(x)dx\right|.
\end{equation} 
 The Fefferman-Stein interpolation theorem in this case can be stated as follows (see e.g. Carbonaro, Mauceri and Meda \cite{CMM2009}). 
\begin{theorem}
Let $G$ be a compact Lie group.  For every $\theta\in (0,1),$ we have,
\begin{itemize}
    \item If $p_\theta=\frac{2}{1-\theta},$ then $(L^2(G),{BMO}(G))_{\theta}=L^{p_\theta}(G).$ 
    \item If $p_\theta=\frac{2}{2-\theta},$ then $({H}^{1}(G),L^2(G))_{\theta}=L^{p_\theta}(G).$ 
\end{itemize}
\end{theorem}
\section{Proof of the main theorems}\label{Section:3}

\subsection{Local estimates for kernels} An important feature in the proof of Theorem \ref{weak} is to obtain suitable kernel estimates of the operators. Before we prove these estimates, we would like to present some remarks that apply to this section. 
\begin{remark}
   The system of balls $B(x,R)$ in $G$ is defined by the geodesic distance, meaning 
\begin{equation*}
    B(x,R)=\{y\in G:d(x,y)<R\},
\end{equation*}
where $d(x,y)=|y^{-1}x|_{geod}$. Furthermore, in our proofs, we often make a local argument using the local diffeomorphism $\exp\colon\mathfrak{g}\to G$. For this reason, we mention that there exists a neighborhood of zero and constants $C_1,C_2>0$ such that (see e.g. \cite[Page 2894]{RT:JFA:2011})
\begin{equation}\label{ineq}
    C_1|\exp(Y)|_{geod}\leq|Y|_{euclid}\leq C_2|\exp(Y)|_{geod}.
\end{equation}
From inequality \eqref{ineq}, it follows that when we have balls $B(e,R)$ in $G$ for sufficiently small radius $R$, they correspond to balls around zero in $\mathfrak{g}\cong\mathbb{R}^n$ with some radius $R'$ proportional to $R$. We will denote this relationship by $R\sim R'$. 
\end{remark}
\begin{remark}\label{Remark:diff:op}
   From now on, we fix a strongly admissible collection of difference operators
\[
\Delta = \Delta_Q = \{\Delta_{q_1}, \dots, \Delta_{q_l}\},
\]
that satisfies the Leibniz property. Such a collection always exists (see \cite[Corollary 5.13]{12} and Remark \ref{remarkD} below), and all subsequent statements involving $$ \Delta_q^\alpha:= \Delta_{q^{\alpha_1}_1\cdots q^{\alpha_l}_l }$$ are understood with respect to this fixed system.  We choose this notation to emphasize the order $|\alpha|$ of the difference operator in our further analysis.
\end{remark}
\begin{remark}\label{remarkD}
Matrix components of unitary representations induce difference operators, see \cite{22}. Indeed, if $\xi_{1},\xi_2,\cdots, \xi_{k},$ are  fixed irreducible and unitary  representation of $G$, which not necessarily belong to the same equivalence class, then each coefficient of the matrix
\begin{equation}
 \xi_{\ell}(g)-I_{d_{\xi_{\ell}}}=[\xi_{\ell}(g)_{ij}-\delta_{ij}]_{i,j=1}^{d_{\xi_\ell}},\, \quad g\in G, \,\,1\leq \ell\leq k,
\end{equation} 
that is each function 
$q^{\ell}_{ij}(g):=\xi_{\ell}(g)_{ij}-\delta_{ij}$, $ g\in G,$ defines a difference operator
\begin{equation}\label{Difference:op:rep}
    \mathbb{D}_{\xi_\ell,i,j}:=\mathscr{F}(\xi_{\ell}(g)_{ij}-\delta_{ij})\mathscr{F}^{-1}.
\end{equation}
We can fix $k\geq \mathrm{dim}(G)$ of these representations in such a way that the corresponding  family of difference operators is admissible, it  means that, 
\begin{equation*}
    \textnormal{rank}\{\nabla q^{\ell}_{i,j}(e):1\leqslant \ell\leqslant k \}=\textnormal{dim}(G).
\end{equation*}
To define higher order difference operators of this kind, let us fix a unitary irreducible representation $\xi_\ell$.
Since the representation is fixed we omit the index $\ell$ of the representations $\xi_\ell$ in the notation that will follow.
Then, for any given multi-index $\alpha\in \mathbb{N}_0^{d_{\xi_\ell}^2}$, with 
$|\alpha|=\sum_{i,j=1}^{d_{\xi_\ell}}\alpha_{i,j}$, we write
$$\mathbb{D}^{\alpha}:=\mathbb{D}_{\xi_{\ell},1,1}^{\alpha_{11}}\cdots \mathbb{D}^{\alpha_{d_{\xi_\ell}d_{\xi_\ell}}}_{\xi_\ell,d_{\xi_\ell},d_{\xi_\ell}}
$$ 
for a difference operator of order $|\alpha|$.
\end{remark}
 Let us establish the following lemma, which will be useful in proving the kernel estimates for small radii. Here we will use the notation $\Delta_q^\alpha$ in Remark \ref{Remark:diff:op}.

\begin{lemma}\label{smoothing}
Let $G$ be a compact Lie group of topological dimension $n \ge 1$, and let 
$T \in \Psi^{m}_{\rho,\delta}(G)$ with
\[
m \le -\frac{n}{2}(1-\rho),
\qquad 0 \le \delta \le \rho \le 1, \quad \rho\neq0,\;\delta \ne 1.
\]
Let $\varphi \in C_c^\infty(\mathbb{R}^n)$ be supported in $[\tfrac{1}{2},1]$, define for $t \ge 1$ 
\[
\sigma_t(x,\zeta) = \sigma(x,\zeta)\,\varphi(\langle \zeta \rangle/t).
\]
Then for every $\alpha \in \mathbb{N}^n$ and every $r \in \mathbb{R}$, there exists a constant $C_r > 0$ such that
\[
\sup_{(x,[\zeta])\in G\times\widehat{G}}\|\Delta_q^\alpha \sigma_t(x,\zeta)\|_{\mathrm{op}} \le C_r t^r.
\]
\end{lemma}
\begin{proof}
    By \cite[Lemma 6.8]{12} we have that for every $r\in\mathbb{R}$ there exists some $C_r>0$ such that
    \begin{equation*}
        \|\Delta_{q}^{\alpha}\varphi\left(\frac{\langle\zeta\rangle}{t} \right)I_{d_\zeta}\|_{op}\leq C_rt^r.
    \end{equation*}
    Using the above and the Leibniz rule (see Remark \ref{Remark:diff:op}) we see that
    \begin{equation*}
        \begin{aligned}
            \|\Delta_{q}^{\alpha}\sigma_t(x,\zeta)\|_{op}&=\sum_{\alpha_1+\alpha_2=\alpha}C_{\alpha_1,\alpha_2}||\Delta_{q}^{\alpha_1}\sigma(x,\zeta)||_{op}\|\Delta_{q}^{\alpha_2}\varphi\left(\frac{\langle\zeta\rangle}{t} \right)I_{d_\zeta}\|_{op}\\
            &\leq\sum_{\alpha_1+\alpha_2=\alpha}C_{\alpha_1,\alpha_2}||\Delta^{\alpha_1}_{q}\sigma(x,\zeta)||_{op}\cdot(C_rt^r)\\
            &\leq\sum_{\alpha_1+\alpha_2=\alpha}C_{\alpha_1,\alpha_2}\cdot C_{\alpha_1}\langle\zeta\rangle^{m-\rho|\alpha_1|}\cdot(C_rt^r).
        \end{aligned}
    \end{equation*}
    In the above we have that $m-\rho|\alpha_1|\leq 0$, and $\langle\zeta\rangle\geq1$, therefore $\langle\zeta\rangle^{m-\rho|\alpha_1|}\leq 1$. Since the sum is finite we can find a constant large enough we call this $C_r'$ such that 
    \begin{equation*}
       \sup_{(x,[\zeta])\in G\times\widehat{G}} \|\Delta_{q}^{\alpha}\sigma_t(x,\zeta)\|_{op}\leq C_r't^{r}.
    \end{equation*}
    The proof is complete. 
\end{proof}
Let us now prove the kernel estimates in Theorem \ref{kernel_estimates} in the form of the following theorem.
\begingroup
\renewcommand{\thetheorem}{1.1} 
\setcounter{theorem}{1}  
\begin{theorem}\label{kernel:lemma}
    Let $G$ be a compact Lie group of topological dimension $n\geq 1$. Consider $T\in\Psi^{m}_{\rho,\delta}(G)$, with
    \[
m \le -\frac{n}{2}(1-\rho),
\qquad 0 \le \delta \le \rho \le 1,\ \rho\neq0,\;\delta \neq 1.
\]
    For $R>0$, we have that the kernel of $T$ satisfies the following. For $y\in\overline{B(z, R)}$ we have
    \begin{itemize}
        \item $R\geq 1$ then
        $\int_{B(z, 2R)^c}|K(x,y)-K(x,z)|\diff{x}\leq C$;
         \item $R<1$ then $\int_{B(z, 2R^{\rho})^c}|K(x,y)-K(x,z)|\diff{x}\leq C$;
    \end{itemize} Here the constant $C>0,$ in the above estimates is independent of $R>0.$
\end{theorem}
\endgroup
\begin{proof} Furthermore, since $m_1 \leq m_2$ implies $S^{m_1}_{\rho,\delta}(G) \subseteq S^{m_2}_{\rho,\delta}(G)$, it suffices to prove the result for the sharp order $-\frac{n(1-\rho)}{2}$.

Let us consider the case where $R\geq 1$. First we make the following observation, namely that if $x\in B(z,2R)^c$ then $x\in B(y,R)^c$. This follows simply by the triangle inequality and the fact $d(y,z)\leq R$,
   \begin{equation*}
       d(x,y)\geq d(x,z)-d(y,z)\geq 2R-R=R.
   \end{equation*}
Using the triangle inequality, we estimate 
\begin{equation*}
\begin{aligned}
    &\int_{B(z, 2R)^c}|K(x,y)-K(x,z)|\diff{x}\\
       &\leq\int_{B(z, 2R)^c}|K(x,y)|\diff{x}+\int_{B(z, 2R)^c}|K(x,z)|\diff{x}\\
       &\leq\int_{B(y,R)^c}|K(x,y)|\diff{x}+\int_{B(z, 2R)^c}|K(x,z)|\diff{x}.
\end{aligned}
\end{equation*}
 To estimate the above integrals we note that we are integrating over $B(y,R)^c$ and $B(z,2R)^c$. Therefore we have that $|y^{-1}x|\geq R\geq1$, $|z^{-1}x|\geq 2R\geq2$, so both $y^{-1}x\neq e$ and $z^{-1}x\neq e$. By Proposition \ref{kernel} the map $(x,y^{-1}x)\to K(x,y)$ satisfies the estimate
 \begin{equation*}
     |K(x,y)|\leq C\sup_{\pi\in\widehat{G}}\|\sigma(x,\pi)\|_{S^{m}_{\rho,a,b}}|y^{-1}x|^{-\frac{n+m}{\rho}}.
 \end{equation*}
 In our case $m=-n(1-\rho)/2$. Since $-(n+m)/\rho<0$, and $d(x,y)>R$ where $R\geq 1$, we have that $|y^{-1}x|^{-\frac{n+m}{\rho}}\leq 1$. So there exists some $C_1>0$ independent of $R$ such that $|K(x,y)|\leq C_1$, similarly we can find $C_2>0$ such that $|K(x,z)|\leq C_2$. Finally we can estimate, 
   \begin{equation*}
       \int_{B(z, 2R)^c}|K(x,y)-K(x,z)|\diff{x}\leq 2C|G|.
   \end{equation*}
   Here $C=\max\{C_1,C_2\}$. 
   
   Let us now consider the case where $R<1$. Let $0\leq\varphi\in C^{\infty}_{0}(\mathbb{R}^n)$ be supported in the interval $[\frac{1}{2},1]$, such that
   \begin{equation*}
       \int_{0}^{\infty}\varphi\left(\frac{1}{t}\right)\frac{\diff{t}}{t}= \int_{1}^{2}\varphi\left(\frac{1}{t}\right)\frac{\diff{t}}{t}=1.
   \end{equation*}
   Define $\sigma_t(x,\zeta)=\sigma(x,\zeta)\varphi(\langle\zeta\rangle/t)$, with kernel given by
   \begin{equation*}
       K_t(x,y)=\sum_{[\zeta]\in\widehat{G}}d_\zeta\operatorname{Tr}(\zeta(y^{-1}x)\sigma_t(x,\zeta)).
   \end{equation*}
   We have that 
   \begin{equation}
       K(x,y)=\int_{1}^{\infty}K_t(x,y)\frac{\diff{t}}{t}. 
       \label{kernel}
   \end{equation}
 To prove the above, we start by observing that for $[\zeta]\in\widehat{G}$ fixed, we have
\begin{equation*}
    \int_{1}^{\infty}\varphi\left(\frac{\langle\zeta\rangle}{t}\right)\frac{dt}{t}=\int_{1/\langle\zeta\rangle}^{\infty}\varphi\left(\frac{1}{s}\right)\frac{ds}{s}=\int_{1}^{2}\varphi\left(\frac{1}{s}\right)\frac{ds}{s}=1,
\end{equation*}
where we used $\langle \zeta \rangle \geq 1$ to conclude $[1,2] \subseteq \left[ \frac{1}{\langle \zeta \rangle}, \infty \right)$. We can now write
\begin{equation*}
    \begin{aligned}
        K(x,y)&=\sum_{[\zeta]\in\widehat{G}}d_\zeta\operatorname{Tr}(\zeta(y^{-1}x)\sigma(x,\zeta))\\
    &=\sum_{[\zeta]\in\widehat{G}}d_\zeta\operatorname{Tr}\left(\zeta(y^{-1}x)\sigma(x,\zeta) \int_{1}^{\infty}\varphi\left(\frac{\langle\zeta\rangle}{t}\right)\frac{dt}{t}\right)\\
    &=\sum_{[\zeta]\in\widehat{G}}d_\zeta\int_{1}^{\infty}\operatorname{Tr}(\zeta(y^{-1}x)\sigma_t(x,\zeta))\frac{dt}{t}\\
    &=\int_{1}^{\infty}\sum_{[\zeta]\in\widehat{G}}d_\zeta\operatorname{Tr}(\zeta(y^{-1}x)\sigma_t(x,\zeta))\frac{dt}{t}\\
    &=\int_{1}^{\infty}K_t(x,y)\frac{dt}{t}.
    \end{aligned}
\end{equation*}
It remains to justify the interchange of the sum and the integral. This follows from Fubini’s theorem, because we have that
\begin{equation*}
\begin{aligned}
    |d_\zeta\operatorname{Tr}(\zeta(y^{-1}x)\sigma_t(x,\zeta))|&\leq d_\zeta\|\zeta(y^{-1}x)\|_{HS}\|\sigma_t(x,\zeta)\|_{HS}\\
    &\leq d_\zeta^2\|\zeta(y^{-1}x)\|_{op}\|\sigma_t(x,\zeta)\|_{op}\\
    &= d_\zeta^2\|\sigma_t(x,\zeta)\|_{op}.
\end{aligned}
\end{equation*}
The above argument is analogous to the one used later in (\ref{trace}). Furthermore, we estimate 
\begin{equation*}
\begin{aligned}
    \int_{1}^{\infty}\sum_{[\zeta]\in\widehat{G}}d_\zeta^2||\sigma_t(x,\zeta)\|_{op}\frac{dt}{t}&\leq\int_{1}^{\infty}\sum_{\frac{t}{2}\leq\langle\zeta\rangle\leq t}d_\zeta^2\cdot C_r t^r\frac{dt}{t}\\
    &\leq\int_{1}^{\infty}C_r\cdot t^n\cdot t^r\frac{dt}{t}.
\end{aligned}
\end{equation*}
In the above, we have used Lemma \ref{smoothing} and the fact that $\varphi$ is supported in $[\frac{1}{2},1]$ to apply Theorem \ref{Weyl}. Finally, the last integral is finite if we choose $r<-n$. Next, we start by estimating the following expression 
   \begin{equation*}
\begin{aligned}
    &\int_{1}^{\infty}\int_{B(z, 2R^\rho)^c}|K_t(x,y)-K_t(x,z)|\diff{x}\frac{\diff{t}}{t}\\
       &=\int_{1}^{1/R}\int_{B(z,2R^\rho)^c}|K_t(x,y)-K_t(x,z)|\diff{x}\frac{\diff{ t}}{t}\\
       &+\int_{1/R}^{\infty}\int_{B(z, 2R^\rho)^c}|K_t(x,y)-K_t(x,z)|\diff{x}\frac{\diff{t}}{t}\\
       &\equiv I+II.
\end{aligned}
\end{equation*}
Let us estimate the two expressions above separately. We start with $I$. Let $N>n/2$. Then
\begin{equation*}
\begin{aligned}
\int_{B(z,2R^{\rho})^c} |K_t(x,y)-K_t(x,z)|\,\diff{x}
&\le
\left(
\int_{G} (1+t^{2\rho}d(x,z)^2)^N
\,|K_t(x,y)-K_t(x,z)|^2 \diff{x}
\right)^{1/2} \\
&\qquad\times
\left(
\int_{G}
(1+t^{2\rho}d(x,z)^2)^{-N} \diff{x}
\right)^{1/2}.
\end{aligned}
\end{equation*}
We proceed to estimate the second integral. The mapping $\exp\colon\mathfrak{g}\to G$ is a local diffeomorphism. That is, there exists some $\varepsilon,\varepsilon'\in(0,1)$, such that $\exp^{-1}\colon B(e,\varepsilon)\to B(0,\varepsilon')$ is a diffeomorphism. Let us start by substituting $x=zu$, then we get
    \begin{equation*}
\begin{aligned}
    \int_{G}(1+t^{2\rho}d(x,z)^2)^{-N}\diff{x}&=\int_{G}(1+t^{2\rho}|u|_{geod}^{2})^{-N}\diff{u}\\
    &=\int_{B(e,\varepsilon)}(1+t^{2\rho}|u|_{geod}^{2})^{-N}\diff{u}+\int_{G\setminus B(e,\varepsilon)}(1+t^{2\rho}|u|_{geod}^{2})^{-N}\diff{u}.
\end{aligned}
\end{equation*}
Let us estimate the first integral. Using the fact that the exponential mapping is a local diffeomorphism, we get 
\begin{equation*}
    \begin{aligned}
        \int_{B(e,\varepsilon)}(1+t^{2\rho}|u|_{geod}^2)^{-N}\diff{u}&\leq\int_{B(0,\varepsilon^{\prime})}(1+t^{2\rho}|u|^2_{euclid})^{-N}\diff{u}\\
        &\leq t^{-\rho n}\int_{\mathbb{R}^n}(1+|y|^2)^{-N}\diff{y}=Ct^{-\rho n}.
    \end{aligned}
\end{equation*}
In the last equality, we have used the fact that $N>\frac{n}{2}$ to conclude that the integral $\int_{\mathbb{R}^n}(1+|y|^2)^{-N}\diff{y}<\infty$. The second integral can be estimated as 
\begin{equation*}
    \int_{G\setminus B(e,\varepsilon)}(1+t^{2\rho}|u|^2)^{-N}\diff{u}\leq\int_{G\setminus B(e,\varepsilon)}t^{-2\rho N}|u|^{-2N}\diff{u}\leq t^{-\rho n}\int_{G\setminus B(e,\varepsilon)}\varepsilon^{-2N}\diff{u}
\end{equation*}
$$\leq\varepsilon^{-2N}|G|t^{-\rho n}.$$
We have shown that 
\begin{equation*}
    \left(\int_{G}(1+t^{2\rho}d(x,z)^2)^{-N}\diff{x}\right)^{1/2}\lesssim t^{-\rho n/2}.
\end{equation*}
This implies that
\begin{equation*}
    \int_{B(z,2R^{\rho})^c}|K_t(x,y)-K_t(x,z)|\diff{x}\lesssim
\left(
\int_{G} (1+t^{2\rho}d(x,z)^2)^N
\,|K_t(x,y)-K_t(x,z)|^2 \diff{x}
\right)^{1/2}
\end{equation*}
$$\times  t^{-\rho n/2}.$$
We will now study the expression above. By the Binomial theorem we can write 
\begin{equation*}
    \begin{aligned}
        &\int_{G}(1+t^{2\rho}d(x,z)^2)^N\cdot|K_t(x,y)-K_t(x,z)|^2\diff{x}
        \\&=\sum_{|\alpha|\leq N}C_{\alpha}t^{2\rho\alpha}\int_{G}d(x,z)^{2\alpha}|K_t(x,y)-K_t(x,z)|^2\diff{x}.
    \end{aligned}
\end{equation*}
Let $\alpha$ be fixed so that $|\alpha|\leq N$. Consider the difference operator $\Delta^{\alpha}_{q}$ associated to $q$ that vanishes at $e$ of order $\alpha$. Then, there exist constants $C_1,C_2>0$ such that
\begin{equation*}
    C_1d(x,z)^{\alpha}\leq|q(z^{-1}x)|\leq C_2 d(x,z)^{\alpha}.
\end{equation*}
Therefore, we have to estimate the following $L^2$-norm. Using Parseval's Theorem we see that
    \begin{equation*}
        \begin{aligned}
           & \int_{G}|q(z^{-1}x)(K_t(x,y)-K_t(x,z))|^2\diff{x}=\int_{G}|q(z^{-1}x)(R_t(x,y^{-1}x)-R_t(x,z^{-1}x))|^2\diff{x}\\
            &=\int_{G}|q(x)(R_t(zx,y^{-1}zx)-R_t(zx,x)|^2\diff{x}\\
            &=\sum_{[\zeta]\in\widehat{G}}d_\zeta||\mathcal{F}_G[q(x)(R_t(zx,y^{-1}zx)-R_t(zx,x)](\zeta)||^{2}_{HS}\\
            &=\sum_{[\zeta]\in\widehat{G}}d_\zeta||\Delta_{q}^{\alpha}(\zeta(y^{-1}z)\sigma_t(yx,\zeta))-\Delta_{q}^{\alpha}\sigma_t(zx,\zeta)||^{2}_{HS}.
        \end{aligned}
    \end{equation*}
  Moreover, by the triangle inequality and in view of the identity
\[
\|AB\|_{HS} \le \|A\|_{op}\|B\|_{HS},
\]
choosing \( B = I_{d_\zeta} \), we have
\begin{equation}\label{trace}
    \|B\|_{HS} = \sqrt{\operatorname{Tr}(I_{d_\zeta})}
= \sqrt{d_\zeta}.
\end{equation}

Hence, we obtain the following
    \begin{equation*}
        \begin{aligned}
            &\sum_{[\zeta]\in\widehat{G}}d_\zeta||\Delta_{q}^{\alpha}(\zeta(y^{-1}z)\sigma_t(yx,\zeta))-\Delta_{q}^{\alpha}\sigma_t(zx,\zeta)||^{2}_{HS}\\&\leq\sum_{[\zeta]\in\widehat{G}}d_\zeta^2(||\Delta_{q}^{\alpha}(\zeta(y^{-1}z)\sigma_t(yx,\zeta))||^2_{op}+||\Delta_{q}^{\alpha}\sigma_t(zx,\zeta)||^{2}_{op}).
        \end{aligned}
    \end{equation*}
    Let us note that in the above sum, we sum over $[\zeta]\in\widehat{G}$ for which $\frac{t}{2}\leq\langle\zeta\rangle\leq t$, because otherwise $\sigma_t$ is zero since $\varphi$ is supported in $[\frac{1}{2},1]$. This observation will be used later in the proof. We will now prove that the above norms can be estimated by some constant times $t^r$ for every real number $r$. The parameter $r$ will be determined later. By Lemma \ref{smoothing} we already have that for every $r\in\mathbb{R}$ there exists some $C_r>0$ such that
    \begin{equation*}
        \|\Delta_{q}^{\alpha}\sigma_t(zx,\zeta)\|_{op}\leq C_rt^{r}.
    \end{equation*}
    We will now prove a similar estimate for the other operator norm, however the resulting bound is slightly different. Using Leibniz rule and Lemma \ref{smoothing} we see that
    \begin{equation*}
    \begin{aligned}
         \|\Delta_{q}^{\alpha}(\zeta(y^{-1}z)\sigma_t(yz,\zeta))||_{op}&=\sum_{\alpha_1+\alpha_2=\alpha}C_{\alpha_1,\alpha_2}||\Delta_{q}^{\alpha_1}\zeta(y^{-1}z)||_{op}\|\Delta_{q}^{\alpha_2}\sigma_t(yx,\zeta)||_{op}\\
         &\leq\sum_{\alpha_1+\alpha_2=\alpha}C_{\alpha_1,\alpha_2}||\Delta_{q}^{\alpha_1}\zeta(y^{-1}z)||_{op}\cdot(C_r t^{r}).
    \end{aligned}
    \end{equation*}
To understand $\|\Delta_{q}^{\alpha_1}\zeta(y^{-1}z)\|_{op}$, we will introduce the symbol
    \begin{equation*}
        \tau(x,\pi)=
\begin{cases}
\zeta(x), & \text{if } \pi=\zeta,\\
0,        & \text{otherwise}.
\end{cases}
    \end{equation*}
    Computing its right convolution kernel we see that 
    \begin{equation*}
        R(x,y)=\sum_{[\pi]\in\widehat{G}}d_\pi\operatorname{Tr}(\pi(y)\tau(x,\pi))=d_\zeta\operatorname{Tr}(\zeta(y)\zeta(x)).
    \end{equation*}
    Therefore we have that 
    \begin{equation*}
    \begin{aligned}
        ||\Delta^{\alpha_1}_{q}\tau(x,\pi)||_{op}&=d_\zeta\int_{G}|q(y)|\cdot|\operatorname{Tr}(\zeta(x)\zeta(y))|\cdot||\zeta^{*}(y)||_{op}\diff{y}\\
        &\leq d_\zeta^2\int_{G}|q(y)|\diff{y}\lesssim\langle\zeta\rangle^{n}.
    \end{aligned}
    \end{equation*}
    We have used that 
    \begin{equation*}
        |\operatorname{Tr}(\zeta(x)\zeta(y))|\leq\|\zeta(x)\|_{HS}\|\zeta(y)\|_{HS}=d_\zeta, 
    \end{equation*}
    and that $\|\zeta^*(y)\|_{op}=1$, for every $[\zeta]\in\widehat{G}$. The final inequality follows from \cite[Proposition 10.3.19]{17}. Finally, we can find a constant $C_r>0$ such that 
    \begin{equation*}
        \|\Delta_{q}^{\alpha}(\zeta(y^{-1}z)\sigma_t(yz,\zeta))||_{op}\leq C_rt^r\langle\zeta\rangle^{n}.
    \end{equation*}
    Continuing the estimation we observe  
    \begin{equation*}
    \begin{aligned}
        &\sum_{\frac{t}{2}\leq\langle\zeta\rangle\leq t}d_\zeta^2(||\Delta_{q}^{\alpha}(\zeta(y^{-1}z)\sigma_t(yx,\zeta))||^2_{op}+||\Delta_{q}^{\alpha}\sigma_t(zx,\zeta)||^{2}_{op})\\
        &\leq\sum_{\frac{t}{2}\leq\langle\zeta\rangle\leq t}d_\zeta^2(C_rt^r\langle\zeta\rangle^{2n}+C_r't^r)\leq\sum_{\frac{t}{2}\leq\langle\zeta\rangle\leq t}d_\zeta^2(C_r^{\prime\prime}t^{2n}t^r)\leq \widetilde{C_r}t^rt^{3n}.
    \end{aligned}
    \end{equation*}
    In the final inequalities, we used the assumption $\frac{t}{2}\leq\langle \zeta \rangle \le t$ to bound $\langle \zeta \rangle^{2n}$ by $t^{2n}$. Since $t \ge 1$, this implies $t^{2n} \ge 1$, which allows us to combine the terms in the sum. Finally, we applied Weyl’s eigenvalue theorem, see Theorem \ref{Weyl}. The original sum can now be estimated as
    \begin{equation*}
        \left(\sum_{|\alpha|\leq N}C_\alpha \cdot \widetilde{C_r}\cdot t^{2\rho\alpha}\cdot t^{3n}\cdot t^{r}\right)^{1/2}\cdot t^{-\rho n/2}.
    \end{equation*}
    For any $\alpha$, let $r = -3n - 2\rho \alpha + \rho n - 2$. We have that,
    \begin{equation*}
        \begin{aligned}
            & \left(\sum_{|\alpha|\leq N}C_\alpha \cdot \widetilde{C_r}\cdot t^{2\rho\alpha}\cdot t^{3n}\cdot t^{r}\right)^{1/2}\cdot t^{-\rho n/2}\\
            &= \left(\sum_{|\alpha|\leq N}C_\alpha \cdot \widetilde{C_{r,\alpha}}\cdot t^{2\rho\alpha}\cdot t^{3n}\cdot t^{-3n-2\rho\alpha+\rho n-2}\right)^{1/2}\cdot t^{-\rho n/2}\\
             &= \left(\sum_{|\alpha|\leq N}C_\alpha \cdot \widetilde{C_{r,\alpha}}\cdot t^{\rho n-2}\right)^{1/2}\cdot t^{-\rho n/2}\\
             &= \left(\sum_{|\alpha|\leq N}C_\alpha \cdot \widetilde{C_{r,\alpha}}\right)^{1/2}\cdot t^{-1}
        \end{aligned}
    \end{equation*}
    Since the above sum is finite, there exists a constant $C_N>0$ such that it is bounded by $C_N t^{-1}$. This concludes the estimation of $I$. Next, the second integral $II$ can be estimated as follows. 
\begin{equation*}
\begin{aligned}
    &\int_{B(z, 2R^\rho)^c}|K_t(x,y)-K_t(x,z)|\diff{x}\\
       &\leq\int_{B(y,R^\rho)^c}|K_t(x,y)|\diff{x}+\int_{B(z, 2R^\rho)^c}|K_t(x,z)|\diff{x}.
\end{aligned}
\end{equation*}
Consider the difference operator $\Delta_q$ associated to $q$ that vanishes at $e$ of order $N$. We have
  \begin{equation*}
      \int_{B(y,R^\rho)^c}|K_t(x,y)|\diff{x}=\int_{B(y,R^\rho)^c}\frac{|q(y^{-1}x)K_t(x,y)|}{|q(y^{-1}x)|}\diff{x}.
  \end{equation*}
  Using the Cauchy-Swchartz inequality and the fact that $|q(y^{-1}x)|\asymp|y^{-1}x|^N$ we obtain
  \begin{equation*}
      \int_{B(y,R^\rho)^c}|K_t(x,y)|\diff{x}\lesssim\left(\int_{B(y,R^\rho)^c}|y^{-1}x|^{-2N}\diff{x}\right)^{1/2}\left(\int_{G}|q(y^{-1}x)K_t(x,y)|^2\diff{x}\right)^{1/2}.
  \end{equation*}
  We first examine the integral $\left(\int_{B(y,R^\rho)^c}|y^{-1}x|^{-2N}\diff{x}\right)^{1/2}$. We start by making the change of variables $u=y^{-1}x$. Note that if $x\in B(y,R^\rho)^c$ then $u\in B(e,R^\rho)^c$, because $|u|=|y^{-1}x|\geq R^\rho$. This gives
  \begin{equation*}
      \int_{B(y,R^\rho)^c}|y^{-1}x|^{-2N}\diff{x}=\int_{B(e,R^\rho)^c}|u|^{-2N}\diff{u}.
  \end{equation*}
  We know that $\exp\colon\mathfrak{g}\to G$ is a local diffeomorphism. Meaning there exists some $\varepsilon,\varepsilon'\in(0,1)$, such that $\exp^{-1}\colon B(e,\varepsilon)\to B(0,\varepsilon')$ is a diffeomorphism. 
  If $R^\rho\leq \varepsilon$, then we can estimate
  \begin{equation*}
   \begin{aligned}
       \int_{|u|>R^\rho}|u|^{-2N}_{geod}\diff{u}\leq\int_{R^\rho<|u|\leq\varepsilon}|u|^{-2N}_{geod}\diff{u}+\int_{\varepsilon<|u|\leq\operatorname{diam}(G)}|u|^{-2N}_{geod}\diff{u}.
   \end{aligned}
  \end{equation*}
 Note that the second integral can be estimated as
  \begin{equation*}
      \begin{aligned}
          \int_{\varepsilon<|u|\leq\operatorname{diam}(G)}|u|^{-2N}_{geod}\diff{u}\leq \varepsilon^{-2N}|G|\leq\varepsilon^{-2N}|G|R^{\rho(n-2N)}.
      \end{aligned}
  \end{equation*}
In the last inequality, we used the fact that $N>\frac{n}{2}$, and $R<1$.  To estimate the first integral, we use that the exponential map is a local diffeomorphism and then apply polar coordinates:
 \begin{equation*}
   \begin{aligned}
       \int_{R^{\rho}<|u|\leq\varepsilon}|u|^{-2N}_{geod}\diff{u}&\lesssim\int_{(R^\rho)'<|u|\leq\varepsilon'}|u|^{-2N}_{eucld}\diff{u}\\
       &\asymp
\int_{(R^\rho)'}^{\varepsilon'}\int_{\Omega}r^{-2N}r^{n-1}\diff{\Omega}\diff{r}\\
      &\asymp\int_{(R^\rho)'}^{\varepsilon'}r^{n-2N-1}\diff{r}\\
      &\asymp\frac{r^{n-2N}}{n-2N}\Big|_{(R^\rho)'}^{\varepsilon'}.
   \end{aligned}
  \end{equation*}
   Since this exponent is negative and $(R^{\rho})'<\varepsilon'$, the dominant contribution in the above expression comes from the lower limit. Moreover, since $R^\rho\sim(R^\rho)^{\prime}$ the integral can be estimated by
   \begin{equation*}
       \int_{R^\rho<|u|\leq\varepsilon}|u|^{-2N}_{geod}\diff{u}\lesssim R^{\rho(n-2N)}.
   \end{equation*}
   If $R^\rho>\varepsilon$ then we have
   \begin{equation*}
       \int_{|u|>R^\rho}|u|^{-2N}_{geod}\diff{u}\leq\int_{\operatorname{diam}(G)>|u|>\varepsilon}|u|^{-2N}_{geod}\diff{u}\leq\varepsilon^{-2N}|G|\leq\varepsilon^{-2N}|G|R^{\rho(n-2N)}.
   \end{equation*}
   Again, we are using that $N>\frac{n}{2}$, and $R<1$. In all the cases above, we have shown that
   \begin{equation*}
       \left(\int_{|u|>R^\rho}|u|_{geod}^{-2N}\diff{u}\right)^{1/2}\lesssim R^{\rho(\frac{n}{2}-N)}.
   \end{equation*}
Next we estimate $\left(\int_{G}|q(y^{-1}x)K_t(x,y)|^2\diff{x}\right)^{1/2}.$ Using Parseval's Theorem and the fact that $\varphi$ is supported in $[\frac{1}{2},1]$ we get
\begin{equation*}
     \begin{aligned}
       \left(\int_{G}|q(y^{-1}x)K_t(x,y)|^2\diff{x}\right)^{1/2}
       &=\left(\sum_{[\zeta]\in\widehat{G}}d_\zeta\|\mathcal{F}_G[{q(y^{-1}x)\mathcal{F}_{G}^{-1}\sigma_t(x,\cdot)(y^{-1}x)}]\|_{HS}^2\right)^{1/2}\\
       &=\left(\sum_{\frac{t}{2}\leq\langle\zeta\rangle\leq t}d_\zeta\|\triangle_q\sigma_t(yx,\zeta)\|_{HS}^2\right)^{1/2}\\
       &=\left(\sum_{\frac{t}{2}\leq\langle\zeta\rangle\leq t}d_\zeta\|\triangle_q\sigma_t(yx,\zeta)\cdot I\|_{HS}^2\right)^{1/2}\\
       &\leq\left(\sum_{\frac{t}{2}\leq\langle\zeta\rangle\leq t}d_\zeta^2\|\triangle_q\sigma_t(yx,\zeta)\|_{op}^2\right)^{1/2}\\
       &\leq\left(\sum_{\frac{t}{2}\leq\langle\zeta\rangle\leq t}d_\zeta^2\langle\zeta\rangle^{-n(1-\rho)-2\rho N}\right)^{1/2}\\
       &\lesssim t^{-\frac{n}{2}(1-\rho)-\rho N}\left(\sum_{\frac{t}{2}\leq\langle\zeta\rangle\leq t}d_\zeta^2\right)^{1/2}\\
       &\lesssim t^{-\frac{n}{2}(1-\rho)-\rho N} t^{n/2}=t^{\rho(\frac{n}{2}-N)}.
   \end{aligned}
   \end{equation*}
   This proves that,
   \begin{equation*}
       \int_{B(y,R^\rho)^c}|K_t(x,y)|\diff{x}\lesssim (tR)^{\rho(\frac{n}{2}-N)}.
   \end{equation*}
   In the same way, it can be proved that 
   \begin{equation*}
       \int_{B(z,2R^\rho)^c}|K_t(x,z)|\diff{x}\lesssim (tR)^{\rho(\frac{n}{2}-N)}.
   \end{equation*}
    Finally, combining (\ref{kernel}) with the estimates for $I$ and $II$, we can apply Fubini’s theorem to obtain 
   \begin{equation*}
   \begin{aligned}
        \int_{B(z, 2R^\rho)^c}|K(x,y)-K(x,z)|\diff{x}&\leq\int_{1}^{\infty}\int_{B(z, 2R^\rho)^c}|K_t(x,y)-K_t(x,z)|\diff{x}\frac{\diff{t}}{t}\\
       &\lesssim\left(\int_{1}^{1/R}\frac{1}{t}+\int_{1/R}^{\infty}(tR)^{\rho(\frac{n}{2}-N)}\right)\frac{\diff{t}}{t}\\
       &\leq C. 
   \end{aligned}
   \end{equation*}
   The proof is complete. 
\end{proof}
\subsection{$L^1$-theory of pseudo-differential operators}
We are now ready to prove Theorem \ref{weak}. For the proof, we will use the Calderón–Zygmund decomposition. In Theorem \ref{CZ}, we have the condition $\alpha\gamma>\frac{1}{|G|}\int_{G}|f(x)|\diff{x}$. Note that if this condition does not hold, a weak (1,1)-bound follows trivially, because
\begin{equation*}
    |\{x\in G:|Tf(x)|>\alpha\}|\leq|G|\leq\frac{1}{\alpha\gamma}\int_{G}|f(x)|\diff{x}.
\end{equation*}
\begingroup
\renewcommand{\thetheorem}{1.2}
\begin{theorem}
    Let $G$ be a compact Lie group of topological dimension $n\geq 1$. Let $T\in\Psi^{m}_{\rho,\delta}(G)$, with
    \[
m \le -\frac{n}{2}(1-\rho),
\qquad 0 \le \delta \le \rho \le 1,\ \rho\neq0,\;\delta \neq 1.
\]  
    Then, $T$ is a continuous mapping from $L^1(G)$ to $L^{1,\infty}(G)$.
\end{theorem}
\endgroup
\addtocounter{theorem}{-1}
\begin{proof} Since the case $\rho=1$ falls into the Calder\'on-Zygmund theory in \cite{7}, in our further arguments we  consider that $0<\rho<1.$ Furthermore, since $m_1 \leq m_2$ implies $S^{m_1}_{\rho,\delta}(G) \subseteq S^{m_2}_{\rho,\delta}(G)$, it suffices to prove the result for the sharp order $-\frac{n(1-\rho)}{2}$.
    Let $f\in L^1(G)$ and $\gamma,\alpha>0$ be arbitrary, such that 
    \begin{equation*}
        \alpha\gamma>\frac{1}{|G|}\int_{G}|f(x)|\diff{x}.
    \end{equation*}
     We have that 
     \begin{eqnarray*}
     &|\{x\in G:|Tf(x)|>\alpha\}|\leq|\{x\in G:|Tg(x)|>\alpha/2\}|\\
     &+|\{x\in G:|Tb(x)|>\alpha/2\}|.
\end{eqnarray*}
    Let us estimate both terms separately. First notice that $g\in L^2(G)$, this we can see by using property 1 from Theorem \ref{CZ} because, 
    \begin{equation*}
        \int_{G}|g|^2\diff{x}=\int_{G}|g(x)|\cdot|g(x)|\diff{x}\lesssim\|g\|_{L^\infty}\cdot\|g\|_{L^1}\lesssim\alpha\gamma\|f\|_{L^1}.
    \end{equation*}
    Thus we have
    \begin{equation*}
    \begin{aligned}
        \alpha^2|\{x\in G: |Tg(x)|>\alpha/2\}|&=\int_{\{x\in G:|Tg(x)|>\alpha/2\}}\alpha^2\diff{x}\\
        &\lesssim\int_{G}|Tg(x)|^2\diff{x}\\
        &=\|Tg\|^{2}_{L^2}\lesssim\|g\|^2_{L^2}\lesssim\alpha\gamma\|f\|_{L^1}.
    \end{aligned}
    \end{equation*}
    Above we have used the fact that $T$ is bounded on $L^2$, because clearly $S^{-\frac{n(1-\rho)}{2}}_{\rho,\delta}\subseteq S^{0}_{\rho,\delta}$, allowing us to apply Proposition \ref{L2}. It remains now to estimate $|\{x\in G:|Tb(x)|>\alpha/2\}|$, for this the kernel estimation is essential. We start by writing 
    \begin{equation*}
        \sum_{j=1}^{\infty}b_j=\sum_{r_j<1}b_j+\sum_{r_j\geq 1}b_j=F+R,
    \end{equation*}
    where $F=\sum_{r_j<1}b_j$ and $R=\sum_{r_j\geq 1}b_j$. Furthermore we have that
     \begin{eqnarray*}
     &|\{x\in G:|Tb(x)|>\alpha/2\}|\leq|\{x\in G:|TF(x)|>\alpha/4\}|\\
     &+|\{x\in G:|TR(x)|>\alpha/4\}|.
\end{eqnarray*}
     Now we notice that there exists $k>0$ such that if $x\in G\setminus \Omega^*$ then $x\in B(x_j,2r_j)^c$ for every $j$, where $\Omega^*=\bigcup_{j=1}^{\infty}B(x_j,kr_j)$. Indeed we can find such $k$ because $x\in G\setminus\Omega^*$ implies that $x\in B(x_j,kr_j)^c$ for every $j$. Meaning $d(x,x_j)\geq kr_j>2r_j$ for every $j$, so it suffices to choose some $k>2$. However we will choose $k>4$, the reason for this will be clear later on. Let us now estimate the part corresponding to $R$. We notice that
    \begin{equation*}
        |\{x\in G:|TR(x)|>\alpha/4\}|\leq|\Omega^*|+|\{x\in G\setminus\Omega^*:|TR(x)|>\alpha/4\}|.
    \end{equation*}
    By the properties of the family $\{I_j\}_j$ we can estimate
    \begin{equation*}
        |\Omega^*|\lesssim\sum_{j=1}^{\infty}|I_j|\lesssim\frac{\|f\|_{L^1}}{\alpha\gamma}.
    \end{equation*}
    To estimate $|\{x\in G\setminus\Omega^*: |TR(x)|>\alpha/4\}|$ we work as following. Using that $\int b_j\diff{y}=0$ we can write
    \begin{equation*}
    \begin{aligned}
        Tb_j(x)&=\int_{G}K(x,y)b_j(y)\diff{y}\\
        &=\int_{G}(K(x,y)-K(x,x_j))b_j(y)\diff{y}\\
        &=\int_{I_j}(K(x,y)-K(x,x_j))b_j(y)\diff{y}.
    \end{aligned}
     \end{equation*}
    In the last equality we used $\operatorname{supp}(b_j)\subseteq I_j$. Therefore we get that
    \begin{equation*}
        \int_{G\setminus\Omega^*}|Tb_j(x)|\diff{x}\leq\int_{G\setminus\Omega^*}\int_{I_j}|K(x,y)-K(x,x_j)|\cdot|b_j(y)|\diff{y}\diff{x}.
    \end{equation*}
    Since $x\in G\setminus\Omega^*$ implies that $x\in B(x_j,2r_j)^c$ for every $j$, we finally obtain 
    \begin{equation*}
    \begin{aligned}
        \int_{G\setminus\Omega^*}|TR(x)|\diff{x}&\leq\sum_{r_j\geq 1}\int_{G\setminus\Omega^*}|Tb_j(x)|\diff{x}\\
        &\leq\sum_{r_j\geq 1}\int_{B(x_j,2r_j)^c}\int_{I_j}|K(x,y)-K(x,x_j)|\cdot|b_j(y)|\diff{y}\diff{x}\\
        &\lesssim\sum_{r_j\geq 1}\int_{I_j}|b_j(y)|\diff{y}\lesssim\sum_{j=1}^{\infty}|I_j|\lesssim(\alpha\gamma)^{-1}\|f\|_{L^1},
    \end{aligned}
    \end{equation*}
    where in the last line we used the kernel estimate in Theorem \ref{kernel:lemma} and the properties of the family $\{I_j\}_j,$ see Theorem \ref{CZ} (4). It remains to estimate $|\{x\in G\setminus\Omega^*:|TF|>\alpha/4\}|$. We start by considering a family $\{\phi_j\}_j$ of smooth functions with compact support, for which $\phi_j\geq 0$ and $\int_{G}\phi_j=1$ for all $j$. In order to define this family, we make the following observation. There are some $r,r'\in(0,1)$ such that $\exp^{-1}\colon B(e,r)\to B(0,r')$ is a diffeomorphism. Choose $0<c<1$ such that $c<r$. We can define for every $j$:
    \begin{equation*}
        \phi_j(y)=\frac{1}{|B(e,c\cdot r_j^{1/\rho})|}1_{B(e,c\cdot r_j^{1/\rho})}(y).
    \end{equation*}
    Now we write
    \begin{equation*}
         F=\sum_{r_j<1}b_j=\sum_{r_j<1}b_j*\phi_j+\sum_{r_j<1}(b_j-b_j*\phi_j)=F'+F'',
    \end{equation*}
    where $F'=\sum_{r_j<1}b_j*\phi_j$ and $F''=\sum_{r_j<1}(b_j-b_j*\phi_j)$. This now implies
    \begin{eqnarray*}
     &|\{x\in G\setminus\Omega^*:|TF(x)|>\alpha/4\}|\leq|\{x\in G\setminus\Omega^*:|TF'(x)|>\alpha/8\}|\\
     &+|\{x\in G\setminus\Omega^*:|TF''(x)|>\alpha/8\}|.
\end{eqnarray*}   
    We start by estimating the expression corresponding to $F''$. Notice that
    \begin{equation*}
    \begin{aligned}
        T(b_j-b_j*\phi_j)(x)&=\int_{G}K(x,y)(b_j-b_j*\phi_j)(y)\diff{y}\\
        &=\int_{G}K(x,y)b_j(y)\diff{y}-\int_{G}K(x,z)(b_j*\phi_j)(z)\diff{z}\\
        &=\int_{G}K(x,y)b_j(y)\diff{y}-\int_{G}K(x,z)\int_{G}\phi_j(y^{-1}z)b_j(y)\diff{y}\diff{z}.
    \end{aligned}
    \end{equation*}
    We can change the order of integration in the second term by using Fubini theorem. Combining this with the fact that $\int_{G}\phi_j=1$ we get that the right hand side is
    \begin{equation*}
        \int_{G}\int_{G}(K(x,y)-K(x,z))\phi_j(y^{-1}z)\diff{z}b_j(y)\diff{y}.
    \end{equation*}
    Using the above we get the following estimation
\begin{equation*}
\begin{aligned}
    \alpha \left|\{x \in G \setminus\Omega^* : |TF''(x)| > \alpha/8\}\right| 
    &\lesssim \int_{G \setminus\Omega^*} |TF''(x)| \, \diff{x} \\
    &\lesssim \sum_{r_j < 1} \int_{G \setminus\Omega^*} \int_{G} \int_{G} \Big| K(x,y) - K(x,z) \Big| \\
    &\quad \cdot |\phi_j(y^{-1}z) b_j(y) |\, \diff{z} \, \diff{y} \, \diff{x}.
\end{aligned}
\end{equation*}
Observe that $\operatorname{supp}(b_j)\subseteq I_j$, $\operatorname{supp}(\phi_j)\subseteq B(e,c\cdot r_j^{1/\rho})$ and $x\in G\setminus kI_j$ implies that $d(x,x_j)>4r_j$. Notice now that if $x\in G\setminus\Omega^*$, $y\in I_j$ and $y^{-1}z\in\operatorname{supp}(\phi_j)\subseteq B(e,c\cdot r_j^{1/\rho})$ then
\begin{equation*}
\begin{aligned}
    d(x,z)&\geq d(x,x_j)-d(x_j,y)-d(y,z)\\
    &\geq 4r_j-r_j-r_j^{1/\rho}\geq 4r_j-r_j-r_j=2r_j,
\end{aligned}
\end{equation*}
where in the last inequality we used that $r_j<1$ and $0<\rho<1$. Finally we obtain
\begin{equation*}
    \mathscr{I}=\sum_{j}\int_{I_j}\int_{\operatorname{supp}(\phi_j)}\int_{B(z,2r_j)^c}|K(x,y)-K(x,z)|\cdot|\phi_j(y^{-1}z)|\cdot|b_j(y)|\diff{x}\diff{z}\diff{y}.
\end{equation*}
Using the kernel condition in Theorem \ref{kernel:lemma}, the fact that $\int_{G}\phi_j(y^{-1}z)\diff{z}=1,$ for any $y$ and $j$, and the property of the $b_j$'s in Theorem \ref{CZ} (5), we can estimate $\mathscr{I}$ as follows
\begin{equation*}
     \mathscr{I}= \sum_{j}\int_{I_j}|b_j(y)|\diff{y}\lesssim\|f\|_{L^1}.
\end{equation*}
It remains to estimate $|\{x\in G\setminus\Omega^*:|TF'(x)|>\alpha/8\}|$. By Chebyshev's inequality we have that 
\begin{equation*}
    \alpha^2|\{x\in G\setminus\Omega^*:|TF'(x)|>\alpha/8\}|\lesssim\|TF'\|^{2}_{L^2}.
\end{equation*}
Let $J^\beta=(1+\mathcal{L}_{G})^{\beta/2}=(1-X_1^2-\dots-X_n^2)^{\beta/2}$ denote the Bessel potential of order $\beta$. We choose $\beta=-n(1-\rho)/2<0$, in order that its kernel could be integrable. Using the $L^2$-boundedness of $T$ and $J^{-\beta}$ we see that
\begin{equation*}
    \|TF'\|^{2}_{L^2}=\|TJ^{-\beta}J^\beta F'\|^{2}_{L^2}\lesssim\|J^\beta F'\|^{2}_{L^2}.
\end{equation*}
From the above we conclude that it is enough to prove that $\|J^\beta F'\|^{2}_{L^2}\lesssim\alpha\|f\|_{L^1}$. To do this we split
\begin{equation*}
    J^\beta F'(x)=\sum_{x\sim I_j}J^\beta(b_j*\phi_j)(x)+\sum_{x\not\sim I_j}J^\beta(b_j*\phi_j)(x)=F_1(x)+F_2(x),
\end{equation*}
with $F_1(x)=\sum_{x\sim I_j}J^\beta(b_j*\phi_j)$ and $F_2(x)=\sum_{x\not\sim I_j}J^\beta(b_j*\phi_j)$. Here we write $x\sim I_j$, which means that $x\in I_j$ or $x\in I_{j'}$ for some $I_{j'}$ such that  $I_j\cap I_{j'}\neq\emptyset$. The notation $x\not\sim I_j$ means the opposite of the previous notation. If we denote by $k_\beta$ the right-convolution kernel associated to $J^\beta$ we can write $F_1(x)=\sum_{x\sim I_j}b_j*\phi_j*k_\beta(x)$ and $F_2(x)=\sum_{x\not\sim I_j}b_j*\phi_j*k_\beta(x)$. Let us prove the estimate $\|F_2\|^2_{L^2}\lesssim\alpha\|f\|_{L^1}$. Observe that
\begin{equation*}
\begin{aligned}
    \int_{G}|\sum_{x\not\sim I_j}b_j*\phi_j*k_\beta(x)|\diff{x}&\leq\sum_{x\not\sim I_j}\int_{G}|b_j*\phi_j*k_\beta(x)|\diff{x}\\
    &\leq\sum_{j}\|b_j*\phi_j*k_\beta\|_{L^1}\\
    &\leq\sum_{j}\|b_j\|_{L^1}\|\phi_j*k_\beta\|_{L^1}.
\end{aligned}
\end{equation*}
Note that $J^\beta$ is a pseudo-differential operator of order $\beta$, and consequently, by Proposition \ref{kernel} its kernel satisfies for $x\in G\setminus\{e\},$ the following inequality
\begin{equation*}
    |k_\beta(x)|\lesssim|x|^{-(\beta+n)},\,\quad 0<|x|\lesssim1.
\end{equation*} 
Let us show that $\|k_\beta\|_{L^1}<\infty$. We have $\exp^{-1}\colon B(e,r)\to B(0,r')$ is a diffeomorphism for some $r,r'\in(0,1)$. Then,
\begin{equation*}
    \|k_\beta\|_{L^1}\lesssim\int_{|x|<r}|x|^{-(n+\beta)}\diff{x}+\int_{|x|\geq r}|x|^{-(n+\beta)}\diff{x}.
\end{equation*}
We now show that both integrals are finite. For the first integral we have
\begin{equation*}
    \int_{|x|<r}|x|_{geod}^{-(n+\beta)}\diff{x}\asymp\int_{|x|<r'}|x|^{-(n+\beta)}_{eucld}\diff{x}<\infty.
\end{equation*}
The last integral is less than infinity because $\beta<0$, or $n>n+\beta$. For the second integral we note that
\begin{equation*}
    \int_{\{x\in G:|x|\geq r\}}|x|^{-(n+\beta)}\diff{x}\leq r^{-(n+\beta)}\int_{G}dx<\infty.
\end{equation*}
This means that
\begin{equation*}
    \|\phi_j*k_\beta\|_{L^1}\lesssim\|\phi_j\|_{L^1}\|k_\beta\|_{L^1}=\|k_\beta\|_{L^1}<\infty.
\end{equation*}
Finally we get
\begin{equation*}
    \|F_2\|_{L^1}\lesssim\sum_{j}\|b_j\|_{L^1}\lesssim\|f\|_{L^1}.
\end{equation*}
From the inequality $\|F_2\|^{2}_{L^2}\leq\|F_2\|_{L^1}\|F_2\|_{L^\infty}$ it is enough to show that $\|F_2\|_{L^\infty}\lesssim\gamma\alpha$. Observe that if we consider $j$ such that $x\not\sim I_j$ then
\begin{equation*}
\begin{aligned}
    |b_j*\phi_j*k_\beta(x)|&\leq\int_{I_j}|\phi_j*k_\beta(y^{-1}x)||b_j(y)|\diff{y}\leq\sup_{y\in I_j}|\phi_j*k_\beta(y^{-1}x)|\int_{I_j}|b_j(y)|\diff{y}\\
    &=\sup_{y\in I_j}|\phi_j*k_\beta(y^{-1}x)|\cdot|I_j|\cdot\frac{1}{|I_j|}\int_{I_j}|b_j(y)|\diff{y}.
\end{aligned}
\end{equation*}
Notice that the supremum of the above expression indeed exists, because
\begin{equation*}
    |\phi_j*k_\beta(y^{-1}x)|\leq(\sup|\phi_j|)\cdot\|k_\beta\|_{L^1}<\infty.
\end{equation*}
Now we follow the observation in \cite[Page 15]{5} that for $j$ with $x\not\sim I_j$, one has  that $\phi_j*k_\beta(y^{-1}x)$ is essentially constant over the ball $I_j=B(x_j,r_j)$ and we can estimate
\begin{equation*}
    \sup_{y\in I_j}|\phi_j*k_\beta(y^{-1}x)|\cdot|I_j|\lesssim\int_{I_j}|\phi_j*k_\beta(y'^{-1}x)|\diff{y'}.
\end{equation*}
Furthermore, observe that the positivity of the kernel $k_\beta$, and of $\phi_j$ leads to
\begin{equation*}
\begin{aligned}
    \int_{I_j}|&\phi_j*k_\beta(y'^{-1}x)|\diff{y'}\frac{1}{|I_j|}\int_{I_j}|b_j(y)|\diff{y}\\
    &=\int_{G}|\phi_j*k_\beta(y'^{-1}x)|\left(\frac{1}{|I_j|}\int_{I_j}|b_j(y)|\diff{y}\right)1_{I_j}(y')\diff{y'}\\
    &=\left(\frac{1}{|I_j|}\int_{I_j}|b_j(y)|\diff{y}\cdot1_{I_j}\right)*\phi_j*k_\beta(x).
\end{aligned}
\end{equation*}
Therefore, we obtain
\begin{equation*}
\begin{aligned}
    |F_2(x)|&\leq\sum_{j:x\not\sim I_j}|b_j*\phi_j*k_\beta(x)|\lesssim\sum_{j:x\not\sim I_j}\left(\frac{1}{|I_j|}\int_{I_j}|b_j(y)|\diff{y}\cdot1_{I_j}\right)*\phi_j*k_\beta(x)\\
    &\lesssim\sum_{j:x\not\sim I_j}\gamma\alpha\cdot1_{I_j}*\phi_j*k_\beta(x)=\int_{G}\sum_{j:x\not\sim I_j}\gamma\alpha\cdot1_{I_j}*\phi_j(z)k_\beta(z^{-1}x)\diff{z}\\
    &\lesssim\gamma\alpha\|k_\beta\|_{L^1}\left\|\sum_{j:x\not\sim I_j} 1_{I_j} * \phi_j \right\|_{L^\infty}.
\end{aligned}
\end{equation*}
 By observing that the supports of the functions $1_{I_j}*\phi_j$ have bounded overlaps we have that $\left\|\sum_{j:x\not\sim I_j} 1_{I_j} * \phi_j \right\|_{L^\infty}<\infty$, and that $\|F_2\|_{L^\infty}\lesssim\gamma\alpha$. Finally, it remains to show that $\|F_1\|^2_{L^2}\lesssim\alpha\|f\|_{L^1}$. Let us define
\begin{equation*}
    F^{j}_1(x)=\begin{cases}
   b_j*\phi_j*k_\beta(x), \hspace{10pt}x\in I_j \\
      0,\hspace{60pt} \text{otherwise}.
\end{cases}
\end{equation*}
Then $F_1=\sum_{j}F^{j}_{1}$. Because of the finite overlapping of the balls $I_j$'s, there exists a $M_0\in\mathbb{N}$, such that for any $x\in G$, $F^j(x)\not=0$, for at most $M_0$ values of $j,$ we have 
\begin{equation*}
\begin{aligned}
    \int_{G}|F_1(x)|^2\diff{x}&\leq M_0\sum_{j}\int_{G}|F_1^j(x)|^2\diff{x}\\
    &=M_0\sum_{j}\int_{I_j}|b_j*\phi_j*k_\beta(x)|^2\diff{x}\\
    &\leq M_0\sum_{j}\|b_j\|^2_{L^1}\|\phi_j*k_\beta\|^2_{L^2}\lesssim\sum_{j}\alpha^2|I_j|^2\|\phi_j*k_\beta\|^2_{L^2}.
\end{aligned}
\end{equation*}
We have to compute $\|\phi_j*k_\beta\|_{L^2(G)}^2$. Notice that $\operatorname{supp}(\phi_j)\subseteq B(e,c)\subseteq B(e,r)$, which allows us to use the fact that the map $\exp^{-1}\colon B(e,r)\to B(0,r')$ is a diffeomorphism to make a local calculation. That is also the reason why we made the choice of this $c$ earlier. A similar calculation as in \cite[Theorem 1.1]{5} shows that
\begin{equation*}
    \|\phi_j*k_\beta\|_{L^2}^2\lesssim|I_j|^{-1}.
\end{equation*}
 Finally, we can show that
 \begin{equation*}
\begin{aligned}
    \|F_1\|^2_{L^2}&\leq\sum_{j}\alpha^2|I_j|^2\|\phi_j*k_\beta\|^2_{L^2}\lesssim\sum_{j}\alpha^2|I_j|^2|I_j|^{-1}=\alpha^2\sum_{j}|I_j|\\
    &\lesssim\alpha\gamma\|f\|_{L^1}.
\end{aligned}
 \end{equation*}The proof is complete.
\end{proof} Now, we investigate the $H^1$-$L^1$-boundedness of pseudo-differential operators and present an alternative proof that holds for the complete range $0 \leq \delta \le \rho \le 1, \delta \neq 1$. We make the following observation. Since $\exp^{-1}\colon B(e,r)\to B(0,r')$ is a diffeomorphism for some $0<r,r'<1$, it follows that $|B(e,R)|\leq CR^n$ for all $R\leq r$. Moreover, since $G$ is compact, we know that $\operatorname{diam}(G)<\infty$, implying the existence of $R_{max}>0$ such that for $R>R_{max}$, $G\subseteq B(e,R)$. For radii $R$ with $r<R\leq R_{max}$, we have the inequality
\begin{equation*}
    Cr^n = |B(e, r)| \leq |B(e, R)| \leq |B(e, R_{\max})|.
\end{equation*}
Therefore, we can choose a constant \(C' > 0\) such that
\begin{equation*}
   |B(e, R)| \leq C' R^n \quad \text{for all} \quad r < R \leq R_{\max}. 
\end{equation*}
The above observation will be used in the proof of Theorem \ref{H1:L1}. Next, we prove the following lemma, which will be useful in proving $H^1$-$L^1$-boundedness. 
\begin{lemma}
    Let $G$ be a compact Lie group of topological dimension $n\geq 1$.  Let $T\in\Psi^{m}_{\rho,\delta}(G)$, with $m=-\frac{n}{2}(1-\rho)$, $0\leq\delta\leq\rho\leq 1$, $\rho\neq0,\;$$\delta\neq 1$.
    Then $T$ is a continuous mapping
    \begin{itemize}
        \item from $L^2(G)$ to $L^{2/\rho}(G)$,
        \item from $L^{2/(2-\rho)}(G)$ to $L^{2}(G)$.
    \end{itemize}
\end{lemma}
\begin{proof}
    Let $J$ be the Bessel potential. Notice that $J^{n(1-\rho)/2}T$ and $TJ^{n(1-\rho)/2}$ are bounded on $L^2(G)$ because they have order 0, so we can apply Proposition \ref{L2}. By the Hardy-Littlewood-Sobolev inequality (see e.g. \cite{HLS}), we can conclude that $J^{-n(1-\rho)/2}$ is a continuous mapping from $L^2(G)$ into $L^{2/\rho}(G)$ and from $L^{2/2-\rho}(G)$ into $L^2(G)$. Thus the lemma follows by composition. 
\end{proof}
\begin{theorem}\label{H1:L1}
    Let $G$ be a compact Lie group of topological dimension $n\geq 1$. Let $T\in\Psi^{m}_{\rho,\delta}(G)$, with \[
m = -\frac{n}{2}(1-\rho),
\qquad 0 \le \delta \le \rho \le 1,\ \rho\neq0,\;\delta \neq 1.
\]
Then, $T$ is a continuous mapping from $H^1(G)$ to $L^{1}(G)$.
\end{theorem}
\begin{proof}
    Let $a$ be an $(1,\infty)$-atom. We have to show that $\|Ta\|_{L^1}\lesssim 1$. By the definition of an atom, we know that there exists some ball $B(z,R)$ such that
    \begin{equation*}
        \operatorname{supp}(a)\subseteq B(z,R),\hspace{5pt}\|a\|_{L^\infty}\leq|B(z,R)|^{-1},\hspace{5pt}\int_{B(z,R)}a(x)\diff{x}=0.
    \end{equation*}
    First suppose that $R<1$. Then we have
    \begin{equation*}
        \int_{G}|Ta(x)|\diff{x}\leq\int_{B(z,2R^\rho)}|Ta(x)|\diff{x}+\int_{G\setminus B(z,2R^\rho)}|Ta(x)|\diff{x}=I_1+I_2,
    \end{equation*}
    where $I_1=\int_{B(z,2R^\rho)}|Ta(x)|\diff{x}$ and $I_2=\int_{G\setminus B(z,2R^\rho)}|Ta(x)|\diff{x}$. Let us estimate both integrals separately. Using H\"older's inequality, the continuity of $T$ from $L^{2/(2-\rho)}(G)$ to $L^{2}(G)$, and the above assumptions on $a$, we obtain that
     \begin{equation*}
        \begin{aligned}
        I_1&\leq\left(\int_{B(z,2R^\rho)}\diff{x}\right)^{\frac{1}{2}}\left(\int_{G}|Ta(x)|^2\diff{x}\right)^{\frac{1}{2}}\\
        &\leq|B(z,2R^\rho)|^{\frac{1}{2}}\cdot\left(\int_{B(z,R)}|a(x)|^{\frac{2}{2-\rho}}\diff{x}\right)^{\frac{2-\rho}{2}}\\
        &\leq|B(z,2R^\rho)|^{\frac{1}{2}}\cdot|B(z,R)|^{-\frac{\rho}{2}}=|B(e,2R^\rho)|^{\frac{1}{2}}\cdot|B(e,R)|^{-\frac{\rho}{2}}\leq C.
    \end{aligned}
    \end{equation*}
    Since the mean of $a$ is zero and $\operatorname{supp}(a)\subseteq B(z,R)$, we can write
   \begin{equation*}
    \begin{aligned}
        Ta(x)&=\int_{B(z,R)}K(x,y)a(y)\diff{y}\\
        &=\int_{B(z,R)}(K(x,y)-K(x,z))a(y)\diff{y}.
    \end{aligned}
    \end{equation*}
    Therefore we estimate 
    \begin{equation*}
    \begin{aligned}
        I_2&\leq\int_{G\setminus B(z,2R^\rho)}\int_{B(z,R)}|K(x,y)-K(x,z)|\cdot|a(y)|\diff{y}\diff{x}\\
        &\lesssim\int_{B(z,R)}|a(y)|\diff{y}\lesssim\|a\|_{L^{\infty}}|B(z,R)|\lesssim 1.
    \end{aligned}
    \end{equation*}
    In the above we have used the kernel condition in Theorem \ref{kernel_estimates} for $R<1$, and the bound $\|a\|_{L^{\infty}}\leq |B(z,R)|^{-1}$. Now suppose $R\geq 1$. We do a similar estimation,
    \begin{equation*}
        \int_{G}|Ta(x)|\diff{x}\leq\int_{B(z,2R)}|Ta(x)|\diff{x}+\int_{G\setminus B(z,2R)}|Ta(x)|\diff{x}=I_1+I_2,
    \end{equation*}
    where $I_1=\int_{B(z,2R)}|Ta(x)|\diff{x}$ and $I_2=\int_{G\setminus B(z,2R)}|Ta(x)|\diff{x}$. We estimate $I_1$ by using Hölder's inequality, the $L^2$-continuity of $T$, and the above conditions on the function $a$.
     \begin{equation*}
        \begin{aligned}
        I_1&\leq\left(\int_{B(z,2R)}\diff{x}\right)^{\frac{1}{2}}\left(\int_{G}|Ta(x)|^2\diff{x}\right)^{\frac{1}{2}}\\
        &\leq|B(z,2R)|^{\frac{1}{2}}\cdot\left(\int_{B(z,R)}|a(x)|^2\diff{x}\right)^{\frac{1}{2}}\\
        &\leq|B(z,2R)|^{\frac{1}{2}}\cdot|B(z,R)|^{-\frac{1}{2}}=|B(e,2R)|^{\frac{1}{2}}\cdot|B(e,R)|^{-\frac{1}{2}}\leq C.
    \end{aligned}
    \end{equation*}
    Finally we estimate $I_2$,
   \begin{equation*}
    \begin{aligned}
        I_2&\leq\int_{G\setminus B(z,2R)}\int_{B(z,R)}|K(x,y)-K(x,z)|\cdot|a(y)|\diff{y}\diff{x}\\
        &\lesssim\int_{B(z,R)}|a(y)|\diff{y}\lesssim\|a\|_{L^{\infty}}|B(z,R)|\lesssim 1.
    \end{aligned}
    \end{equation*}
    In the above we used the kernel condition in Theorem \ref{kernel_estimates} for $R\geq 1$, and the bound $\|a\|_{L^{\infty}}\leq |B(z,R)|^{-1}$. The proof is complete. 
\end{proof}
\begin{theorem}\label{L:infty:BMO}
     Let $G$ be a compact Lie group of topological dimension $n\geq 1$. Let $T\in\Psi^{m}_{\rho,\delta}(G)$, with\[
m = -\frac{n}{2}(1-\rho),
\qquad 0 \le \delta \le \rho \le 1,\ \rho\neq0,\;\delta \neq 1.
\]
Then, $T$ is a continuous mapping from $L^{\infty}(G)$ to $BMO(G)$.
\end{theorem}
\begin{proof}
  By the duality argument and Theorem \ref{H1:L1} the result follows.
\end{proof}

\section{Applications to subelliptic estimates on $SU(2)\cong \mathbb{S}^3$}\label{applications}
In this section we consider an application of our main Theorem \ref{weak}.
Let us consider the compact Lie group $G=SU(2)\cong \mathbb{S}^3$ and let us consider a basis of its Lie algebra $\mathfrak{g}=\mathfrak{su}(2).$ Let $X,$ $Y, Z=[X,Y]=XY-YX,$ be three left-invariant vector fields, orthonormal with respect to the Killing form, and let us consider $T$ to be: 
\begin{itemize}
    \item (i) the sub-Laplacian in $SU(2),$ namely, $T=\mathcal{L}_{sub}=X^2+Y^2,$ or,
    \item  (ii) $T=Z-\mathcal{L}_{sub}$ is the analogue of the heat operator in $SU(2).$
\end{itemize}  We are interested in obtaining {\it end-point $L^1$-subelliptic estimates} for $T,$ of the form (and then to prove that for any $N>0,$ there exists $C_N>0,$ such that),
\begin{equation}\label{subelliptic}
    \Vert u\Vert_{L^{1,\infty}(SU(2))}\leq C \Vert Tu\Vert_{W^{1,s}}+C_N\Vert u\Vert_{W^{1,-N}}.
\end{equation}
Here, we have denoted by $W^{p,\beta}$ the space of distributions $u$ on $G=SU(2),$ such that 
\begin{equation}
    \Vert u\Vert_{W^{p,\beta}}=\Vert J^\beta u\Vert_{L^p},\,\,1\leq p\leq \infty,
\end{equation} with $J^\beta=(1+\mathcal{L})^\frac{\beta}{2},$ being the Bessel potential of order $\beta$ associated to the positive Laplacian $\mathcal{L}=-X^2-Y^2-Z^2.$

Our goal is to compute the loss of regularity $s=s_T$  in \eqref{subelliptic}, in order to use such an {\it a-priori estimate,} to understand the subelliptic problem
\begin{equation}\label{IVP} \begin{cases}Tu=f ,& \text{ }
\\u,f\in \mathscr{D}'(SU(2)):=(C^\infty(SU(2)))'. & \text{ } \end{cases}
\end{equation} From the point of view of the theory of pseudo-differential operators, for obtaining the estimate \eqref{subelliptic} it is enough to understand the mapping properties of  the {\it parametrix } $T^{\sharp}$ of $T,$ and then the problem \eqref{subelliptic} falls in the terminology of the harmonic analysis, to prove the weak (1,1)-boundedness of pseudo-differential operators on  $SU(2).$ Due to the fact that the principal symbol  of $T,$ in the sense of H\"ormander \cite{Hormander1985III} is not invariantly defined under changes of coordinates on $G=SU(2)\cong \mathbb{S}^3$, see \cite{9}, the classical machinery of the pseudo-differential calculus defined in terms of the Fourier transform on $\mathbb{R}^n$, and extended to compact manifolds by changes of coordinates, cannot be applied.

Among other things, for $1<p<\infty,$ the third author and Delgado have investigated  the $L^p$-subelliptic estimate,
\begin{equation}\label{subellipticLp}
    \Vert u\Vert_{L^{p}(SU(2))}\leq C \Vert Tu\Vert_{W^{p,s}}+C_N\Vert u\Vert_{W^{p,-N}},
\end{equation} computing $s$ in terms of $p>1.$ For $p=1$,
 we now have the following result.
\begin{theorem}\label{application}The following subelliptic estimates hold (that is, for any $N>0,$ there exists $C_N>0$ such that):
\begin{equation}
    \Vert u\Vert_{L^{1,\infty}}\leq C \Vert Tu\Vert_{W^{1,-\frac{1}{4}}}+C_N\Vert u\Vert_{W^{1,-N}},\quad  T=\mathcal{L}_{sub},
\end{equation} and,
\begin{equation}
    \Vert u\Vert_{L^{1,\infty}}\leq C \Vert Tu\Vert_{W^{1,-\frac{1}{4}}}+C_N\Vert u\Vert_{W^{1,-N}},\quad  T=Z-\mathcal{L}_{sub}.
\end{equation}
    
\end{theorem}
\begin{proof}
   Note that the parametrix $T^\sharp$ of $T,$ belongs to the class $\Psi^{-1}_{\frac{1}{2},0}(SU(2)),$ see \cite{19}. Note that there exists a regularising operator $R\in \Psi^{-\infty}(SU(2))=\bigcap_m \Psi^{m}_{1,0}(SU(2)),$ such that 
    \begin{equation}
        T^\sharp T u =u+Ru.
    \end{equation} In consequence
    \begin{equation}
        \Vert u+Ru\Vert_{L^{1,\infty}} =\Vert T^\sharp T u\Vert_{L^{1,\infty}}
        = \Vert T^\sharp(1+\mathcal{L})^{\frac{1}{8}} (1+\mathcal{L})^{-\frac{1}{8}}Tu\Vert_{L^{1,\infty}}.
    \end{equation} Since the operator $T^\sharp(1+\mathcal{L})^{  \frac{1}{8}}\in \Psi^{-\frac{3}{4}}_{\frac{1}{2},0}$ has order $m=-1+\frac{1}{4}=-3/4,$ and since $\rho=\frac{1}{2}$, and $\delta=0$, Theorem \ref{weak} implies that $T^\sharp(1+\mathcal{L})^{\frac{1}{8}}$ is of weak (1,1) type. In consequence
    \begin{equation}
        \Vert T^\sharp(1+\mathcal{L})^{\frac{1}{8}} (1+\mathcal{L})^{-\frac{1}{8}}Tu\Vert_{L^{1,\infty}}\leq \Vert  (1+\mathcal{L})^{-\frac{1}{8}}T u\Vert_{L^{1}}=\Vert Tu \Vert_{W^{1,-\frac{1}{4}}}.
    \end{equation} So, we have proved that 
    \begin{equation}
         \Vert u\Vert_{L^{1,\infty}}\leq \Vert Tu \Vert_{W^{1,-\frac{1}{4}}}+  \Vert Ru\Vert_{L^{1,\infty}}.
    \end{equation} On the other hand, since $R$ is smoothing, we have that, for every $N>0,$ there exists $C_N>0,$ such that
    \begin{equation}
        \Vert Ru\Vert_{L^{1,\infty}}\leq \Vert Ru\Vert_{L^{1}}\leq C_N\Vert u\Vert_{W^{1,-N}}.
    \end{equation}The proof is complete.   
\end{proof} We have the following result.
\begin{corollary} Let us consider $T$ as above acting on $C^\infty(SU(2))$ and let us consider the subelliptic problem,
    \begin{equation}\label{IVP} \begin{cases}Tu=f ,& \text{ }
\\u,f\in \mathscr{D}'(SU(2)):=(C^\infty(SU(2)))'. & \text{ } \end{cases}
\end{equation} If $f\in W^{1,-\frac{1}{4}},$ then $u\in L^{1,\infty}.$
\end{corollary}
\begin{proof}
    Since $u\in D'(SU(2))$ and $T$ is hypoelliptic, in view of the sum of squares theorem due to H\"ormander \cite{13}, we have that $u\in C^\infty(SU(2)).$ In consequence, $u\in W^{2,-N}(SU(2))$ for any $N\in \mathbb{N}.$ Choosing an arbitrary $N\in \mathbb{N},$ we have that 
    \begin{equation}
        \Vert u\Vert_{  W^{1,-N} }\leq \Vert u\Vert_{  W^{2,-N} }<\infty.
    \end{equation} Since $u\in  W^{1,-N}  $ and $f=Tu\in W^{1,-\frac{1}{4}},$ the subelliptic estimate in Theorem \ref{application} implies that $u\in L^{1,\infty}.$ The proof is complete.
\end{proof}

\bibliographystyle{amsplain}

\begin{thebibliography}{99}
\bibitem{Weyl} Akylzhanov, R., Ruzhansky, M., Net spaces on lattices, Hardy-Littlewood type inequalities, and their converses. \textit{Eurasian Math. J.}, 8(3), 10–27, (2017).
    \bibitem{1} Alvarez, J., Hounie, J., Estimates for the kernel and continuity properties of pseudo-differential operators, \textit{Ark. Mat.} 28(1-2), 1--22, (1990).
    \bibitem{2} Alvarez, J., Milman, M., Vector-valued inequalities for strongly singular Calderón-Zygmund operators, \textit{Rev. Mat. Iberoamericana} 2, 405--426, (1986).
    \bibitem{3} Beals, R., $L^p$ and Holder estimates for pseudodifferential operators: necessary conditions, in \textit{Harmonic Analysis in Euclidean Spaces}, American Mathematical Society, Providence, RI, pp. 153--157, (1979).
    \bibitem{4} Beals, R., $L^p$ and Holder estimates for pseudodifferential operators: sufficient conditions, \textit{Ann. Inst. Fourier} 29(3), 239--260, (1979).
    \bibitem{CMM2009} Carbonaro, A., Mauceri, G., Meda, S., $H^1$ and $BMO$ for certain locally doubling metric measure spaces, \textit{Ann. Sc. Norm. Super. Pisa Cl. Sci.} (5) 8, 543--582, (2009).
    \bibitem{Cardona2014} Cardona, D. Weak type (1,1) bounds for a class of periodic pseudo-differential operators,
    \textit{J. Pseudo-Differ. Oper. Appl.}, Vol. 5(4), 507-515, (2014).
    \bibitem{Cardona:Martinez} Cardona, D., Martinez, M. A., Boundedness of pseudo-differential operators on the torus revisited, I. \textit{J. Math. Anal. Appl.,} 554, Vol. 2, No. 129959, (2025).  
    \bibitem{CR} Cardona, D., Ruzhansky, M., Subelliptic pseudo-differential operators and Fourier integral operators on compact Lie groups, to appear in \textit{MSJ Memoirs, Mathematical Society of Japan, Tokyo}, 180 pages. arXiv:2008.09651.
    \bibitem{5} Cardona, D., Ruzhansky, M., Oscillating singular integral operators on compact Lie groups revisited, \textit{Math. Z.} 303, 26 (2023).
    \bibitem{6} Chen, J., Fan, D., Central oscillating multipliers on compact Lie groups, \textit{Math. Z.} 267, 235--259, (2011).
    \bibitem{7} Coifman, R., Weiss, G., Analyse harmonique non-commutative sur certains espaces homogènes (French), Étude de certaines intégrales singulières, in \textit{Lecture Notes in Mathematics}, vol. 242, pp. v+160, Springer, Berlin--New York (1971).
    \bibitem{HLS} Coulhon, T., Saloff-Coste, L., Varopoulos, N., \textit{Analysis and Geometry on Groups}, Cambridge Tracts in Mathematics, 100, Cambridge University Press, Cambridge (1993).
    \bibitem{8} Dasgupta, A., Ruzhansky, M., Gevrey functions and ultradistributions on compact Lie groups and homogeneous spaces, \textit{Bull. Sci. Math.} 138, 756--782, (2014).
    \bibitem{9} Delgado, J., Ruzhansky, M., $L^p$-bounds for pseudo-differential operators on compact Lie groups, \textit{J. Inst. Math. Jussieu} 18(3), 531--559, (2019).
    \bibitem{10} Fefferman, C., Inequalities for strongly singular integral operators, \textit{Acta Math.} 24, 9--36, (1970).
    \bibitem{11} Fefferman, C., $L^p$ bounds for pseudo-differential operators, \textit{Israel J. Math.} 14, 413--417, (1973).
    \bibitem{12} Fischer, V., Intrinsic pseudo-differential calculi on any compact Lie group, \textit{J. Funct. Anal.} 268, 3404--3477, (2015).
    \bibitem{Hi} Hirschman, I., Multiplier transformations I, \textit{Duke Math. J.} 26, 222--242, (1956).
    \bibitem{Hormander1985III} H\"ormander, L. { The Analysis of the linear partial differential operators} Vol. III. Springer-Verlag, (1985)
    \bibitem{13} Hörmander, L., Pseudo-differential operators and hypoelliptic equations, in \textit{Proceedings Symposium on Singular Integrals}, American Mathematical Society, Providence, RI pp. 138--183, (1967).
    \bibitem{14} Hörmander, L., On the $L^2$-continuity of pseudo-differential operators, \textit{Communications on Pure and Applied Math.} 24, 529--535, (1971).
    \bibitem{15} Hounie, J., On the $L^2$-continuity of pseudo-differential operators, \textit{Communications in Partial Differential Equations} 11(7), 765--778, (1986).
    \bibitem{16} Journé, J. L., Calderón-Zygmund operators, pseudo-differential operators and the Cauchy integral of Calderón, \textit{Lecture Notes in Math.}, 994, Springer-Verlag (1983).
    \bibitem{17} Ruzhansky, M., Turunen, V., \textit{Pseudo-differential Operators and Symmetries: Background Analysis and Advanced Topics}, Birkhäuser, Basel (2010).
    \bibitem{18} Ruzhansky, M., Turunen, V., Global quantization of pseudo-differential operators on compact Lie groups, $\operatorname{SU}(2)$, 3-sphere, and homogeneous spaces, \textit{Int. Math. Res. Not. IMRN} 2013(11), 2439--2496, (2013).
    \bibitem{RT:JFA:2011} Ruzhansky, M., Turunen, V., Sharp Gårding inequality on compact Lie groups, \textit{J. Funct. Anal.} 260, 2881--2901, (2011).
    \bibitem{19} Ruzhansky, M., Turunen, V., and Wirth, J., Hormander class of pseudo-differential operators on compact Lie groups and global hypoellipticity, \textit{J. Fourier Anal. Appl.} 20, 476--499, (2014).
    \bibitem{20} Ruzhansky, M., Wirth, J., Global functional calculus for operators on compact Lie groups, \textit{J. Funct. Anal.} 267, 144--172, (2014).
    \bibitem{21} Ruzhansky, M., Wirth, J., On multipliers on compact Lie groups, \textit{Funct. Anal. Appl.} 47(1), 87--91, (2013).
    \bibitem{22} Ruzhansky, M., Wirth, J., $L^p$ Fourier multipliers on compact Lie groups, \textit{Math. Z.} 280(3--4), 621--642, (2015).
    \bibitem{23} Shubin, M. A., \textit{Pseudo-differential Operators and Spectral Theory}, second edition, Springer, Berlin, translated from the 1978 Russian original by Stig I. Andersson (2001).
    \bibitem{24} Stein, E. M., Interpolation of linear operators, \textit{Trans. Amer. Math. Soc.} 83, 482--492, (1956).
    \bibitem{25} Stein, E. M., \textit{Topics in Harmonic Analysis Related to the Littlewood-Paley Theory}, Vol. 63 of Annals of Mathematics Studies, Princeton University Press, New Jersey (1970).
    \bibitem{26} Stein, E. M., Weiss, G., \textit{Introduction to Fourier Analysis on Euclidean Spaces}, No. 32, Princeton Mathematical Series, New Jersey (1971).
    \bibitem{27} Tao, T., The weak-type (1,1) of Fourier integral operators of order $-(n-1)/2$, \textit{J. Aust. Math. Soc.} 76(1), 1--21, (2004).
    \bibitem{28} Taylor, M., \textit{Pseudodifferential Operators}, Princeton University Press, New Jersey (1981).
    \bibitem{Wg} Wainger, S., Special trigonometric series in $k$-dimensions, \textit{Mem. Amer. Math. Soc.} No. 59, 102 (1965).
    \bibitem{29} Weiss, N., $L^p$ estimates for bi-invariant operators on compact Lie groups, \textit{Amer. J. Math.} 94, 103--118, (1972).
    \bibitem{30} Zygmund, A., \textit{Trigonometric Series I and II}, third edition, Cambridge University Press, Cambridge (2003).
\end{thebibliography}

\end{document}